\documentclass{amsart}
\usepackage{color}
\usepackage{amssymb}
\usepackage{amsmath}
\usepackage{amsthm}
\usepackage{amscd}
\usepackage{mathrsfs}
\usepackage{tikz}
\usepackage{graphicx}
\usepackage{subcaption}
\usepackage{bm}
\theoremstyle{definition}
\newtheorem{dfn}{Definition}[section]
\theoremstyle{plain}
\newtheorem{thm}[dfn]{Theorem}

\newtheorem{prop}[dfn]{Proposition}
\newtheorem{cor}[dfn]{Corollary}
\newtheorem{lm}[dfn]{Lemma}
\theoremstyle{remark}
\newtheorem{exam}[dfn]{Example}
\newtheorem{rem}[dfn]{Remark}
\usepackage[linktocpage,colorlinks,linkcolor=blue,anchorcolor=blue,citecolor=blue,urlcolor=black,pagebackref]{hyperref}
\title{The spectra of polynomials in free (semi)circular operators}
\author{Akihiro Miyagawa}
\address{Department of Mathematics, Kyoto University, Kitashirakawa Oiwake-cho, Sakyo-ku, 606-8502, Japan}
\email{miyagawa.akihiro.8p@kyoto-u.ac.jp}
\begin{document}
\begin{abstract}
    We show that any $L^2$-bounded rational function in free semicircular random variables is a bounded operator, which implies the coincidence of the usual spectrum and $L^2$-spectrum for rational functions. Based on this observation, we also compute the spectra of several polynomials in free circular random variables.   
\end{abstract}
\maketitle

\section{Introduction}
For an operator $X$, the spectral set $\mathrm{spec}(X)$ of $X$ is the set of $\lambda \in \mathbb{C}$ such that $(\lambda-X)^{-1}$ does not exist as a bounded operator.
Computing the spectrum of a polynomial in free random variables is one of the basic problems in free probability theory since, in the self-adjoint case, this is equivalent to computing the support of its spectral measure. In connection with Random Matrix theory,
Haageurp and Thorbj\o rnsen \cite{MR2183281}
proved the almost sure convergence of the operator norms of polynomials in independent Gaussian Unitary Ensembles (GUE) towards those in freely independent semicircular random variables. This phenomenon is called the strong convergence. By continuous function calculus, the strong convergence implies almost sure convergence of the spectra of self-adjoint polynomials with respect to the Hausdorff distance. Recently, there have been many developments on the strong convergence phenomenon with applications to other fields of Mathematics (see a survey by van Handel \cite{vanhandel2026strongconvergencephenomenon}). However, as written in \cite[Section 6.9]{vanhandel2026strongconvergencephenomenon}, the spectra of non-self-adjoint polynomials are poorly
understood. 

In this paper, we investigate the spectra of (not necessarily self-adjoint) polynomials in free semicircular and circular random variables. An easy, but important observation is that, for a polynomial $P$ in free semicircles, $(\lambda-P)^{-1}$ is a ``rational function" in a non-commutative sense. Let $D(s)$ be the set of such rational functions (see Section \ref{Preliminaries} for a precise definition).  These rational functions are realized as closed operators affiliated with the von Neumann algebra generated by free semicircles, which are usually unbounded operators. Our first main theorem shows that any $L^2$-bounded rational function in free semicircles is indeed a bounded operator.
\begin{thm}\label{Main}
    If $f \in D(s) \cap L^2(s,\tau)$, then $f \in L^{\infty}(s,\tau)$.
\end{thm}
Once we prove the theorem above, we have the following corollary, which tells us the coincidence of the usual spectrum and ``$L^2$-spectrum".
\begin{cor}\label{spectrum}
    Let $f \in D(s)$. Then, the following three sets are equal: 
    \begin{align*}
       \mathrm{spec}(f)
        &= \mathrm{spec}^2_1(f) =\{\lambda \in \mathbb{C}|\ (\lambda-f)^{-1}\notin L^2(s,\tau)\}
    \end{align*} 
    where $\mathrm{spec}(f)$ is the spectral set of $f$ and $\mathrm{spec}^2_1(f)$ is the complement of the set of $\lambda_0\in \mathbb{C}$ such that there is a neighborhood $N_{\lambda_0}$ of $\lambda_0$ such that $(\lambda-f)^{-1}$ is in $L^2(s,\tau)$ for any $\lambda \in N_{\lambda_0}$.
\end{cor}
\begin{proof}
Obviously, we have the following inclusions,
 \begin{align*}
       \mathrm{spec}(f)^{c}
        \subset \mathrm{spec}^2_1(f)^c \subset \{\lambda \in \mathbb{C}|\ (\lambda-f)^{-1}\in L^2(s,\tau)\}.
    \end{align*} 
For the converse, we may assume $f \notin \mathbb{C}$ ($f \in \mathbb{C}$ is a trivial case). Since $f \in D(s)$ and $D(s)$ is a skew field, $(\lambda-f)^{-1}\in D(s)$ for any $\lambda \in \mathbb{C}$. By Theorem \ref{Main}, we have $D(s) \cap L^{2}(s,\tau)=D(s) \cap L^{\infty}(s,\tau)=C_{\mathrm{div}}(s)$, which concludes the proof.
\end{proof}
 The notion of the $L^p_n$-spectrum $\mathrm{spec}^p_n(f)$ was introduced by Kemp and Hall \cite[Definition 6.1]{MR3998225} for the study of the support of the Brown measures. The Brown measure \cite{MR866489} is a probability measure defined for any (not necessarily normal) operator in a tracial von Neumann algebra, which is considered as a natural extension of the spectral measure. It is known that the usual spectrum $\mathrm{spec}(f)$ of any operator $f$ contains the support of its Brown measure \cite[Proposition 2.17]{MR2339369}. 
Moreover, Zhong \cite[Theorem 4.6]{MR5025950} proved $L^2_1$-spectrum $\mathrm{spec}^2_1(f)$ of any operator $f$ contains the support of its Brown measure. However, the support of the Brown measure of a non-normal operator is not necessarily equal to its spectrum. 
A typical example of such a non-normal operator is a $\mathscr{R}$-diagonal element whose distribution is characterized by the product $uh$ of a free pair of the Haar unitary $u$ and positive operator $h$. Haagerup and Larsen proved that the support of the Brown measure of a $\mathscr{R}$-diagonal element is an annulus.
 \begin{thm}[Theorem 4.4 in \cite{MR1784419}]\label{Haagerup-Larsen}
    Let $a$ be $ \mathscr{R}$-diagonal element in a $\mathrm{W}^*$-probability space. Then, the support of the Brown measure of $a$ is
    \[ \{\lambda \in \mathbb{C}|\ \|a^{-1}\|^{-1}_2 \le \lambda \le \|a\|_2 \}.\]
    
 \end{thm}
By this theorem, if we take a $\mathscr{R}$-diagonal element which is not invertible as a bounded operator, but is invertible as an $L^2$-operator, then there exists a gap between the spectrum and the support of the Brown measure of $a$.
As far as known examples, such a gap often comes from the gap between the usual spectrum $\mathrm{spec}(f)$ and $L^2_1$-spectrum $\mathrm{spec}^2_1(f)$. 
In the case of free semicircular operators, our main result shows that there is no gap between the usual spectrum and $L^2_1$-spectrum, which motivates us to explore the following question: 
\begin{itemize}
    \item Is there any gap between the spectrum and the support of the Brown measure of a rational (polynomial) function in free semicircular operators?
\end{itemize}
 In the second part of this paper, we compute the spectra of several types of holomorphic polynomials (polynomials without adjoint operators): homogeneous polynomials, free random walk models in \cite{driver2025matrixrandomwalkslima}, and quadratic polynomials in freely independent Voiculescu's circular operators. For homogeneous polynomials, free random walk models, we see the coincidence of their spectrum and the support of their Brown measure. For quadratic polynomials, we have the following theorem (${}^t\! b$ denotes the transpose of $b$).
\begin{thm}\label{quadratic}
  For a quadratic polynmial in free circular random variables $c_1,c_2$ written as $f= \sum_{i,j= 1}^2 a_{ij} c_i c_j +\sum_{i=1}^2 b_i c_i$ with $A=(a_{ij})\in M_2(\mathbb{C})$, $b={}^t\!\begin{pmatrix}
      b_1&b_2
  \end{pmatrix}\in \mathbb{C}^2$, the spectral set of $f$ is the set of $\lambda \in \mathbb{C}$ such that $\lambda=0$ or $\lambda$ satisfies \[\lim_{n\to \infty}Q_{\lambda}^ne_1\neq0,\] where $e_1={}^t\begin{pmatrix}
      1&0&0&0&0&0
  \end{pmatrix}$ and $Q_{\lambda}$ is the following $6\times 6$ matrix  :
  \[Q_{\lambda}=\begin{pmatrix}
      b_{\lambda}^* b_\lambda& \mathrm{Tr}(|A_{\lambda}|^2)& b_{\lambda}^*A_{\lambda} & {}^t\! b_{\lambda} \overline{A_{\lambda}}\\
      1 & 0& 0& 0 \\
      \overline{b_{\lambda}}&0 &0 & \overline{A_{\lambda}}\\
      b_{\lambda} & 0 & A_{\lambda} & 0
  \end{pmatrix}\]
  with $A_{\lambda}=\lambda^{-1}A$ and $b_{\lambda}=\lambda^{-1}b$.
  \end{thm}

\begin{rem}\label{equiv_sepctral1}
    In many cases, the condition $\lim_{n\to \infty}Q_{\lambda}^ne_1\neq0$ is equivalent to the condition that the spectral radius of $Q_{\lambda}$ is greater than or equal to $1$. From the Jordan standard form of $Q_{\lambda}$, we can see that if $\lim_{n\to \infty}Q_{\lambda}^ne_1\neq0$, then the spectral radius $Q_{\lambda}$ is greater than or equal to $1$. Since we take a specific vector $e_1$, the converse statement might not be true, although we couldn't find a counterexample. 
    Actually, in the following cases, we can replace $\lim_{n\to \infty}Q_{\lambda}^ne_1\neq0$ by the spectral radius $r(Q_{\lambda})\ge 1$:
     \begin{itemize}
      \item $b=0$. 
      \item $\overline{A}A$ doesn't have different real eigenvalues.
      \item $A$ is symmetric.
  \end{itemize}
  When $P=\frac{1}{2}\sum_{i,j=1}^2c_ic_j+\frac{1}{2}\sum_{i=1}^2 c_i$ and $P=c_1c_2+c_2c_1+c_1+c_2$ ((A), (B) in Figure \ref{fig:four_images}), the corresponding matrices $A$ are
  \[ \frac{1}{2}\begin{pmatrix}
      1&1\\1&1
  \end{pmatrix}, \quad \begin{pmatrix}
      0&1\\1&0
  \end{pmatrix}\]
which are symmetric.  

When $P=c_1^2+c_2c_1-c_2^2+c_1+\sqrt{-1}c_2$ ((C) in Figure \ref{fig:four_images}), the corresponding matrix $A$ is
\[ \begin{pmatrix}
      1&0\\1&-1
  \end{pmatrix}\]
which is not symmetric, but  
\[ A\overline{A}=\begin{pmatrix}
      1&0\\0&1
  \end{pmatrix}\]
  has only one eigenvalue $1$.
  When $P=\frac{{\sqrt{-1}}}{2}c_1^2+c_1c_2+2c_2c_1+c_2^2$ ((D) in Figure \ref{fig:four_images}), then we have $b=0$. Therefore, $\lim_{n\to \infty}Q_{\lambda}^ne_1\neq0$ is equivalent to $r(Q_{\lambda})\ge 1$ for all examples in Figure \ref{fig:four_images}. 
    We discuss them at the end of this paper (Remark \ref{equivalence_spectral2}). 
\end{rem}

 In Random Matrix theory, the Brown measure of a non-normal operator has an important role. For example, it is known that the empirical eigenvalue distribution of the Ginibre ensemble almost surely converges in distribution to the uniform measure on the unit disc, which is the Brown measure of a circular operator \cite{MR2722794}. This result is extended to the case of quartic polynomials in the independent Ginibre ensembles in \cite{MR4492979}.     
 However, the weak and strong convergence of the empirical eigenvalue distributions for general polynomials in the independent Ginibre ensembles is known as one of the open problems in \cite[Section 6.9]{vanhandel2026strongconvergencephenomenon}. Our numerical computations of Theorem \ref{quadratic} in Figure \ref{fig:four_images} seem to be consistent with the strong convergence phenomenon. According to the argument in \cite[Section 6]{brailovskaya2025eigenvaluesbrownianmotionsmathrmglnmathbbc}, the coincidence of the support of the Brown measure and the spectral set seems to be one of the keys to this problem. Therefore, we hope that our results in this paper give some insights into answering this question.

\clearpage
  
\begin{figure}[htbp]
  \centering
  \begin{subfigure}{0.48\columnwidth}
    \centering
    \includegraphics[width=\textwidth]{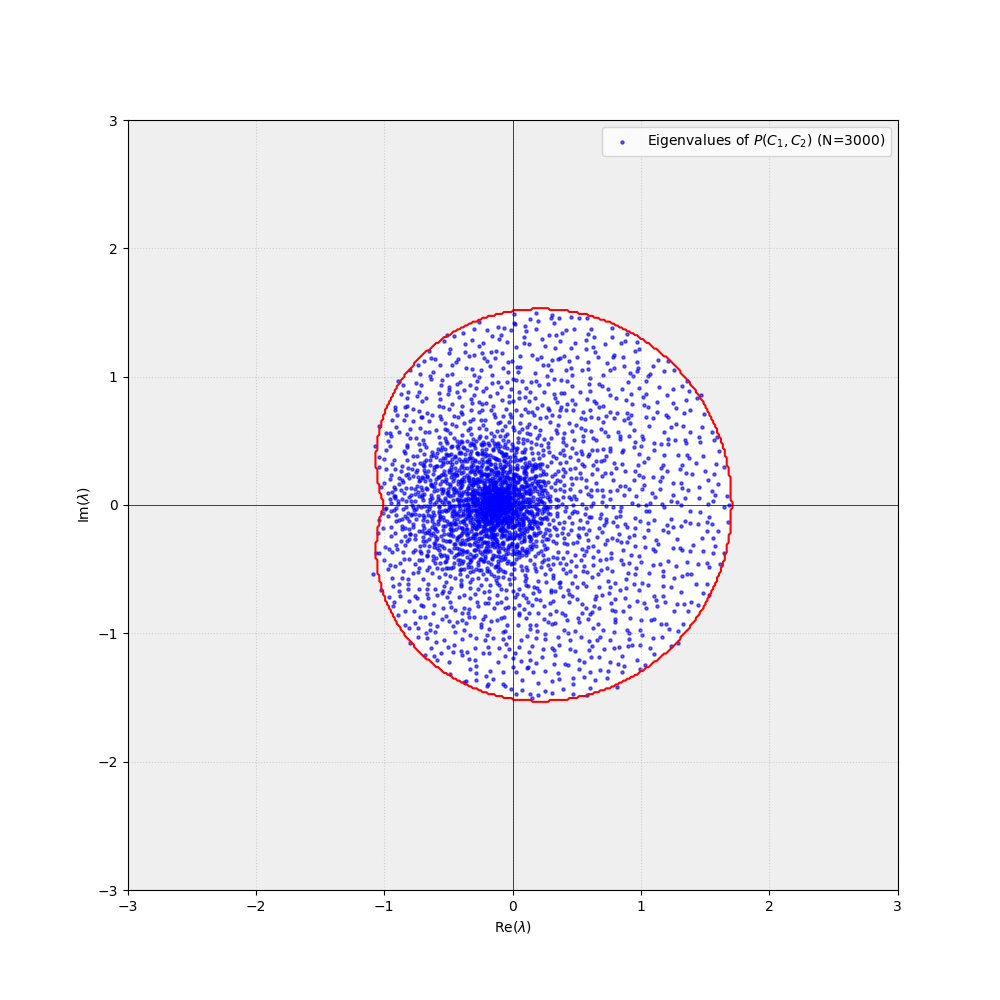}
    \caption{$P=\frac{1}{2}\sum_{i,j=1}^2c_ic_j+\frac{1}{2}\sum_{i=1}^2 c_i$}
  \end{subfigure}
  \hfill 
  \begin{subfigure}{0.48\columnwidth}
    \centering
    \includegraphics[width=\textwidth]{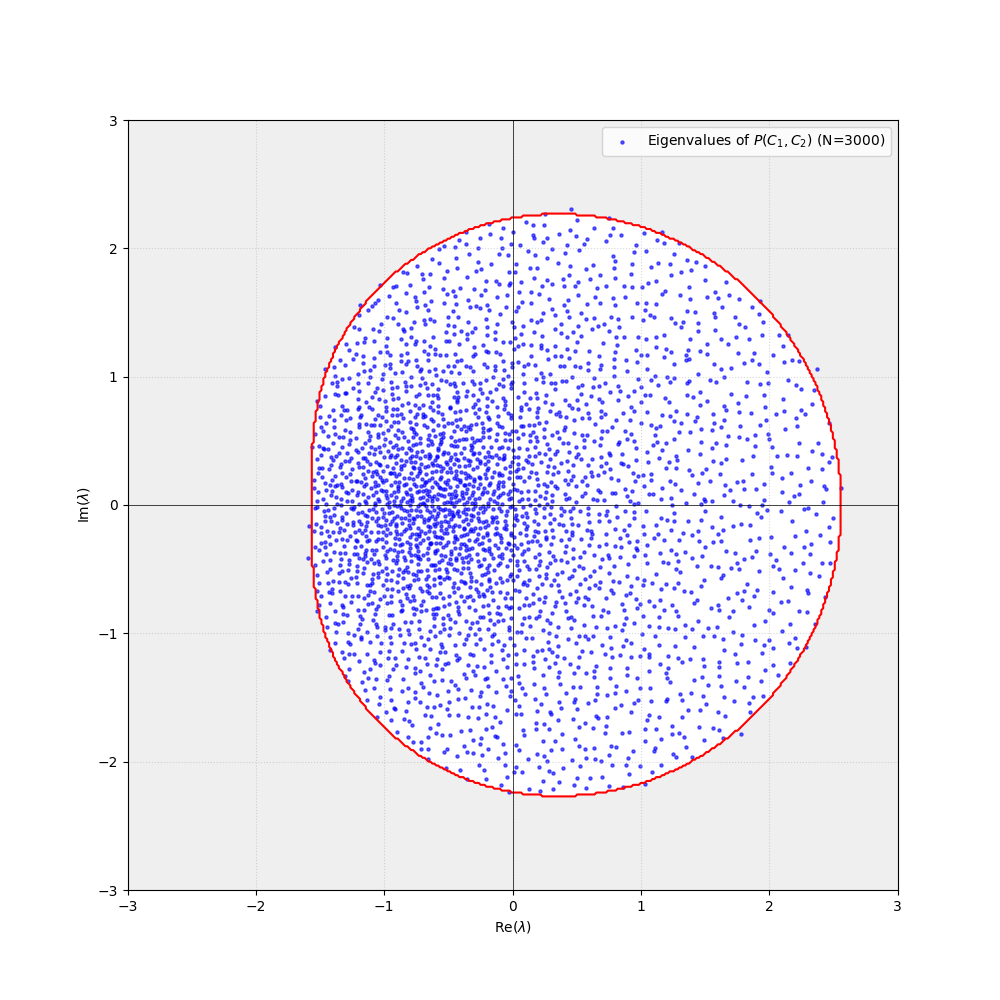}
    \caption{$P=c_1c_2+c_2c_1+c_1+c_2$}
  \end{subfigure}

  \vspace{1em} 
  
  \begin{subfigure}{0.48\columnwidth}
    \centering
    \includegraphics[width=\textwidth]{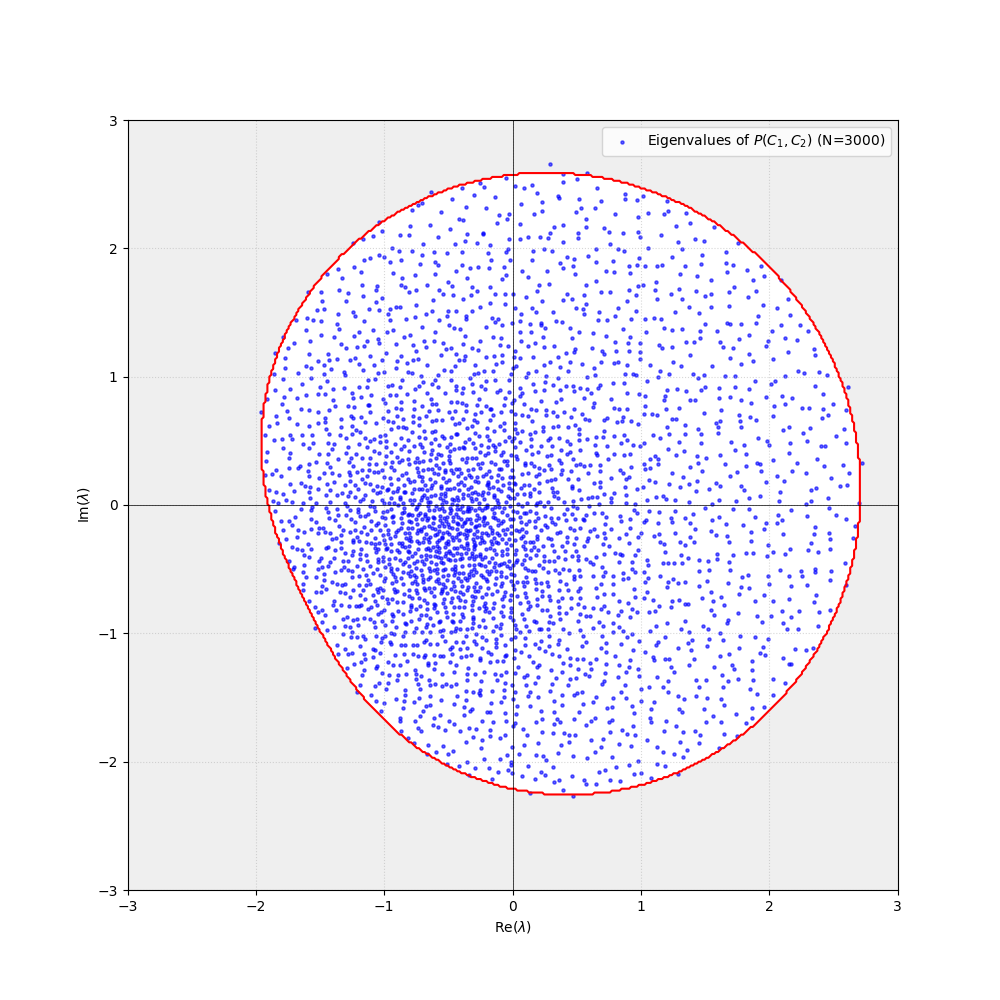}
    \caption{$P=c_1^2+c_2c_1-c_2^2+c_1+\sqrt{-1}c_2$}
  \end{subfigure}
  \hfill
  \begin{subfigure}{0.48\columnwidth}
    \centering
    \includegraphics[width=\textwidth]{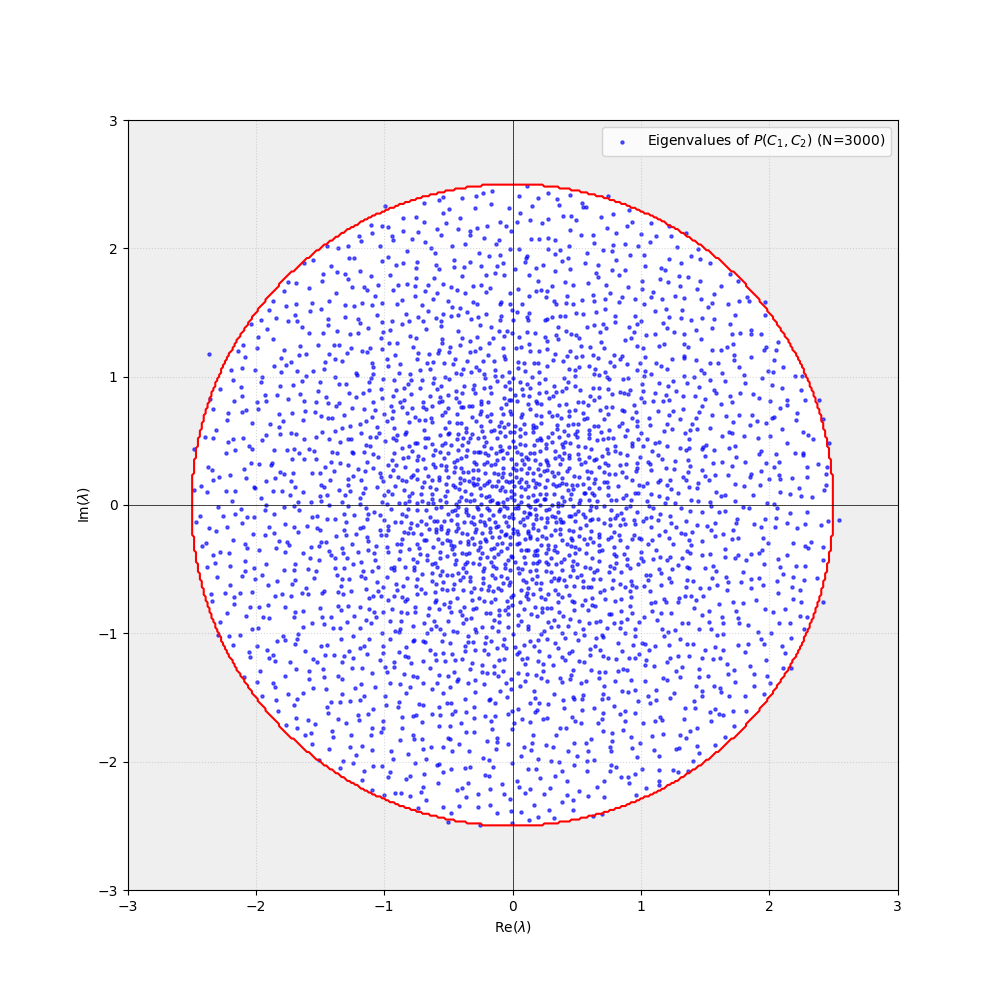}
    \caption{$P=\frac{{\sqrt{-1}}}{2}c_1^2+c_1c_2+2c_2c_1+c_2^2$}
  \end{subfigure}

   \caption{For each polynomial $P$, the area of $\lambda \in \mathbb{C}$ where the spectral radius $r(Q_{\lambda})<1$ is colored in gray, and the red line describes its boundary. The eigenvalues of corresponding Ginibre models of size $N=3000$ are plotted in blue.}
  \label{fig:four_images}
\end{figure}

\section{Preliminaries}\label{Preliminaries}
In this section, we recall fundamental tools for the study of rational functions in free semicircular operators. We refer \cite{MR4452070} for details.

Let $d \in \mathbb{N}$ and set $H=\mathbb{C}^d$. We consider the algebraic Fock space \[\mathcal{F}_{\mathrm{alg}}(H)=\mathbb{C}\Omega\oplus \bigoplus_{n=1}^{\infty} H^{\otimes n}\]
which is a finite linear span of the unit vector $\Omega$ and $H^{\otimes n}$ $(n\in \mathbb{N})$. Let $\mathcal{F}(H)$ denote the full Fock space, which is a completion of $\mathcal{F}_{\mathrm{alg}}(H)$ with respect to the usual tensor norm. For $\xi \in H$, we define the left and right creation operators $l(\xi), r(\xi)$ by
\begin{align*}
    l(\xi)(\xi_1\otimes \xi_2 \otimes \cdots \xi_n)&=\xi \otimes \xi_1\otimes \xi_2 \otimes \cdots \otimes \xi_n\\
    r(\xi)(\xi_1\otimes \xi_2 \otimes \cdots \xi_n)&=\xi_1\otimes \xi_2 \otimes \cdots \otimes \xi_n \otimes \xi.
\end{align*}
The adjoint operators $l^*(\xi)$ and $r^*(\xi)$ are called the left and right annihilation operators, and they satisfy
\begin{align*}
 l^*(\xi)\Omega &= 0,\\
     l^*(\xi)(\xi_1\otimes \xi_2 \otimes \cdots \xi_n)&=\langle\xi_1,\xi\rangle \xi_2 \otimes \cdots \otimes \xi_n,\\
     r^*(\xi)\Omega& = 0,\\
    r^*(\xi)(\xi_1\otimes \xi_2 \otimes \cdots \xi_n)&=\langle \xi_n , \xi\rangle \xi_1\otimes \xi_2 \otimes \cdots \otimes \xi_{n-1}. 
\end{align*}
Let $\{e_i\}_{i=1}^d$ be an orthonormal basis of $H=\mathbb{C}^d$ and set 
\[ s_i = l(e_i)+l^*(e_i).\]
We call the $d$-tuple $s=(s_1,\ldots,s_d)$ the standard free semicircular system. Actually, $s=(s_1,\ldots,s_d)$ has the same joint distribution as free semicircular elements  with respect to the vacuum state defined by (\cite[Section 2.6]{MR1217253})
\[ \tau(X)=\langle X\Omega , \Omega\rangle.\]

Let $\mathbb{C}\langle s \rangle$ denote the $\ast$-algebra generated by $s=(s_1,\ldots,s_d)$ and $L^{\infty}(s,\tau)$ denote the von Neumann algebra generated by $s$ (i.e. the closure of $ \mathbb{C}\langle s \rangle$ in weak operator topology), which is a $\mathrm{II}_1$-factor \cite[Theorem 2.6.2]{MR1217253} 
with the unique tracial state $\tau$. 
Let $L^2(s,\tau)$ be the completion of $\mathbb{C}\langle s \rangle$ with respect to the $L^2$-norm $\|X\|_2^2=\tau(X^*X)=\|X\Omega\|_{\mathcal{F}(H)}^2$ induced by $\tau$, which is a GNS Hilbert space of $L^{\infty}(s,\tau)$.
We remark that the linear map $\mathbb{C}\langle s \rangle \ni X \mapsto X\Omega \in \mathcal{F}_{\mathrm{alg}}(H) $ extends a Hilbert space isomorphism between $L^2(s,\tau)$ and $\mathcal{F}(H)$. 

 Since we will deal with unbounded operators, under the standard representation, we consider the set $L^0(s,\tau)$ of closed operators affiliated with $L^{\infty}(s,\tau)$ which is a $\ast$-algebra (\cite[Theorem XV]{MurrayVonNeumann1936}). 
We can identify $f \in L^2(s,\tau)$ with a (square-integrable) closed operator $L_f \in L^0(s,\tau)$ defined by the closure of the following densely defined closable operator (see, for example, \cite[Theorem 7.3.2]{AnantaramanPopa}):
\[ L_{f} (\xi)= R_{\xi}(f) , \quad \xi \in \widehat{L^{\infty}(s,\tau)}\subset L^2(s,\tau)\]
where $R_{\xi}$ denotes a right multiplication by $\xi$ (recall that $L^2(s,\tau)$ is a $L^{\infty}(s,\tau)$ -$L^{\infty}(s,\tau)$ bimodule).
Note that the domain of $L_{f}$ contains $\mathbb{C}\langle s \rangle $ since $ \mathbb{C}\langle s \rangle \subset L^{\infty}(s,\tau)$. Via the isomorphism $L^2(s,\tau)\cong \mathcal{F}(H)$, we often see $L_{f}$ as a closed operator on $\mathcal{F}(H)$ whose domain contains $\mathcal{F}_{\mathrm{alg}}(H)=\mathbb{C}\langle s \rangle \Omega$.

In this paper, we consider ``rational functions" in the standard semicircular system. One of the way to define ``rational functions" is algebraically taking a closure of the set of polynomials $\mathbb{C}\langle s \rangle$ so that the closure is closed under taking inverse. Such closure is called the division closure (see \cite[Definition 4.10]{MAI2023110016}). Let $C_{\mathrm{div}}(s)$ denote the division closure of $\mathbb{C}\langle s \rangle$ in $L^{\infty}(s,\tau)$, i.e., the smallest subalgebra of $L^{\infty}(s,\tau)$ which contains $\mathbb{C}\langle s \rangle$ and is closed under taking inverse (i.e., $f^{-1} \in C_{\mathrm{div}}(s)$ for $ f \in C_{\mathrm{div}}(s)$ if $f^{-1} \in L^{\infty}(s,\tau)$).
   Let $D(s)$ denote the division closure of $\mathbb{C}\langle s \rangle$ in $L^0(s,\tau)$, i.e., the smallest subalgebra of $L^0(s,\tau)$ which contains $\mathbb{C}\langle s \rangle$ is closed under taking inverse (i.e., $f^{-1} \in D(s)$ for $f \in D(s)$ if $f^{-1} \in L^0(s,\tau)$), which forms a skew field \cite{MAI2023110016}. The main idea comes from the following theorem proved by the author \cite{MR4452070}, which is a free semicircle analog of \cite{MR1452435} and \cite{MR1742860} for free Haar unitaries.
\begin{thm}[Theorem 3.1 and 4.5 in \cite{MR4452070}]\label{keythm}
For $f \in L^{\infty}(s,\tau)$, we have the following equivalence,
    \[ f \in C_{\mathrm{div}}(s) \Longleftrightarrow \{[r^*(e_i),f]\}_{i=1}^d \ \mathrm{are \ finite \ rank \ operators.}\]
    Moreover, we have
    \[ D(s)\cap L^{\infty}(s,\tau)=C_{\mathrm{div}}(s).\]
\end{thm}

\section{Proof of Main theorem}
In this section, we prove Theorem \ref{Main}. We start with $f\in D(s) \cap L^2(s,\tau)$. The following lemma is the first step to prove $f \in L^{\infty}(s,\tau)$.
\begin{lm}\label{lemma1}
   For $f\in D(s) \cap L^2(s,\tau)$, $[r^*(e_i),f]\mathcal{F}_{\mathrm{alg}}(H)$ is finite dimensional for all $i\in [d]$.
\end{lm}
\begin{proof}
Note that $[r^*(e_i),f]$ is well-defined on $\mathcal{F}_{\mathrm{alg}}(H)$ since $\mathcal{F}_{\mathrm{alg}}(H)$ is an invariant subspace for $r^*(e_i)$ and $f \in L^2(s,\tau)$. 
Since $D(s)$ is a skew field, $(1+ff^*)^{-1}f\in D(s) \cap L^{\infty}(s,\tau)=C_{\mathrm{div}}(s)$ by Theore \ref{keythm}. In particular, $[r^*(e_i),(1+ff^*)^{-1}f]$ is a finite rank operator for all $1\le i \le d$.
Since $f(\xi)\in \mathcal{F}(H)$ is well-defined for $\xi \in \mathcal{F}_{\mathrm{alg}}(H)$ and $(1+ff^*)^{-1}\in L^{\infty}(s,\tau)\subset B(\mathcal{F}(H))$, we have $ [(1+ff^*)^{-1}f](\xi)=(1+ff^*)^{-1}[f(\xi)].$ 
We also have the following identity on $\mathcal{F}_{\mathrm{alg}}(H)$ based on the Leibniz rule for commutators: 
    \[ [r^*(e_i),(1+ff^*)^{-1}f]=(1+ff^*)^{-1}[r^*(e_i),f]+[r^*(e_i),(1+ff^*)^{-1}]f.\]
   Note that $(1+ff^*)^{-1}\in D(s) \cap L^{\infty}(s,\tau)$ and $[r^*(e_i),(1+ff^*)^{-1}]$ is a finite rank operator. Thus, the image of $(1+ff^*)^{-1}[r^*(e_i),f]$ is also finite dimensional. By injectivity of $(1+ff^*)^{-1}$ (because $(1+ff^*)(1+ff^*)^{-1}=1$), the image of $[r^*(e_i),f]$ is also finite dimensional. 
\end{proof}
The remaining part is almost the same as the proof of Proposition 3.9 in \cite{MR4452070}, but we give a self-contained proof. Let $[d]^*$ be the free monoid generated by $[d]=\{1,\ldots,d\}$ and the identity (empty word) $\Omega$. For each word $w=w_1w_2\cdots w_n$ ($w_i \in [d]$), we set $|w|=n$. If $f$ belongs to $L^2(s,\tau)$, we can write $f(\Omega) = \sum_{w \in [d]^*}\alpha_w e_w$ where $\alpha_w \in \mathbb{C}$ with $\|f\|^2_2=\sum_{w \in [d]^*} |\alpha_w|^2<\infty$ and $e_w = e_{w_1}\otimes \cdots \otimes e_{w_n} \in H^{\otimes n}$ for $w=w_1\cdots w_n \in [d]^*$ ($e_{\Omega}:=\Omega$). We set $w^*=w_nw_{n-1}\cdots w_1$ for $w=w_1w_2\cdots w_n$. Note that $\{e_w\}_{w \in [d]^*}$ forms an orthonormal basis of $\mathcal{F}(H)$. 
\begin{lm}\label{lemma2}
  For $f\in D(s) \cap L^2(s,\tau)$ with $f(\Omega)=\sum_{w \in [d]^*}\alpha_w e_w$, a linear subspace $K\subset \mathcal{F}(H)$ spanned by \[\left\{\sum_{u \in [d]^*}|\alpha_{uv}|^2e_u \Bigg|\ v \in [d]^*\right\}\] is finite dimensional.
\end{lm}
\begin{proof}
    By using the formula in \cite[Lemma 3.3]{MR4452070}, we have
   \[ [r^*(e_i),f]e_v= \sum_{u \in [d]^*}\alpha_{uiv^*}e_u .\]
   Lemma \ref{lemma1} tells us the linear subspace $K'$ spanned by \[\left\{\sum_{u \in [d]^*}\alpha_{uv}e_u \Bigg|\ v \in [d]^*\right\}\] is finite dimensional. 
   Let $g_i=\sum_{w \in [d]^*} \beta^i_w e_w$ ($i=1,\ldots,k$) be a basis of $K'$. We set 
   \[ g_{ij}=\sum_{w \in [d]^*}\beta^i_w \overline{\beta^j_w} e_w\in \mathcal{F}(H)\]
   for $i,j \in [k]$. Then, for each $v \in [d]^*$, there are $\{c_i\}_{i=1}^k \subset \mathbb{C}$ such that
   \[\sum_{u \in [d]^*}\alpha_{uv}e_u=  \sum_{u \in [d]^*} \left(\sum_{i=1}^k c_i\beta^i_u\right) e_u.\]
   Since $\{e_u\}_{u\in [d]^*}$ is an orthonormal basis, by comparing coefficients, we have
   \[\alpha_{uv}=\sum_{i=1}^k c_i\beta^i_u.  \]
   By applying this identity to $|\alpha_{uv}|^2$, we obtain
   \begin{align*}
       \sum_{u \in [d]^*}|\alpha_{uv}|^2e_u &=\sum_{u \in [d]^*} \left(\sum_{i=1}^k c_i\beta^i_u\right)\overline{\left(\sum_{j=1}^k c_j\beta^j_u\right)} e_u\\
       &=\sum_{i,j=1}^k c_i\overline{c_j} \sum_{u \in [d]^*}\beta^i_u \overline{\beta^j_u} e_u \\
       &=\sum_{i,j=1}^k c_i\overline{c_j} g_{ij}.
   \end{align*}
   Therefore, the linear subspace $K$ is spanned by $\{g_{ij}\}_{i,j\in[k]}$, and $K$ is finite-dimensional.
    \end{proof}

\begin{lm}\label{keylemma}
    Let $f=\sum_{n=0}^{\infty}a_n e_n\in c_0(\mathbb{Z}_{\ge 0})$. If the linear subspace $K$ of $c_0(\mathbb{Z}_{\ge 0})$ spanned by $\sum_{n=0}^{\infty}a_{n+m} e_n$ for all $m \in \mathbb{Z}_{\ge 0}$ is finite dimensional, then we have
    \begin{enumerate}
     \item There exists $k\in \mathbb{N}$, $B\in M_k(\mathbb{C})$, and $\lambda,\gamma \in \mathbb{C}^k$ such that we have for any $n \in \mathbb{Z}_{\ge 0}$,
     \[ a_n={}^t\!\lambda B^n \gamma \]
     \item There exists $M>0$ and $0<A<1$ such that 
     \[ |a_n|\le MA^n.\]
    \end{enumerate}
\end{lm}
\begin{proof}
Let $g_1,\ldots,g_k$ be a basis of $K$ and $l,l^*$ be the shift operators on $c_0(\mathbb{Z}_{\ge 0})$ defined by $le_n=e_{n+1}$, $l^*e_0=0$ and $l^* e_n=e_{n-1}$ for $n \in \mathbb{N}$. For $f \in c_0(\mathbb{Z}_{\ge 0})$ and $n \in \mathbb{Z}_{\ge 0}$, we define $\langle f, e_n\rangle $ by the $n$-th component of $f$. Note that $l^{*m}f=\sum_{n=0}^{\infty} a_{n+m} e_n$. Since $K$ is an invariant subspace of $l^*$, there exists a matrix $B=(b_{ij})\in M_k(\mathbb{C})$ such that
    \[ l^*g_i=\sum_{j=0}^k b_{ij} g_j.\]
    We can also take $\lambda={}^t\! (c_1,\ldots,c_k)\in \mathbb{C}^k$ so that
    \[f=\sum_{i=1}^k  c_i g_i. \]
Then, we have
\begin{align*}
    a_n&=\langle l^{*n} f, e_0\rangle\\
    &=\sum_{i=0}^k c_i  \langle l^{*n}g_{i},e_0\rangle \\
    &= \sum_{i=0}^k \sum_{j_1=0}^k c_i b_{ij_1}\langle l^{*n-1}g_{j_1},e_0\rangle\\
    &=\cdots\\
    &=\sum_{i=0}^k \sum_{j_1,\ldots,j_n=0}^k c_i b_{ij_1}b_{j_1j_2}\cdots b_{j_{n-1}j_n} \langle g_{j_n},e_0\rangle\\
    &={}^t\!\lambda B^n \gamma
\end{align*}
where $\gamma={}^t\!(\langle g_1,e_0\rangle,\ldots,\langle g_k,e_0\rangle)\in \mathbb{C}^k$.
Now, we show 
\[ {}^t\! \lambda \mathbb{C}[B]={}^t\!\mathbb{C}^k, \ \mathbb{C}[B]\gamma=\mathbb{C}^k\]
where $\mathbb{C}[B]$ is the set of polynomials in $B$. For the first claim, we take $P_i$ such that $P_i(l^*)f=g_i$ (recall $g_i \in K$). Then, we have
\begin{align*}
\sum_{j=1}^k [{}^t\!\lambda P_i(B)]_j g_j&=\sum_{l=1}^k c_lP_i(B)_{lj}g_j\\
&=\sum_{l=1}^k c_l P_i(l^*)g_l\\
&=P_i(l^*)f\\
&=g_i.
\end{align*}
Since $\{g_i\}_{i=1}^{k}$ is a basis, we can see $\{{}^t\! \lambda P_i(B)\}_{i=1}^k$ is the canonical basis of ${}^t\! \mathbb{C}^k$.
On the other hand, for a polynomial $P$, we have
\[[P(B)\gamma]_i=\sum_{j=1}^k \langle P(B)_{ij}g_j,e_0\rangle=\langle P(l^*)g_i,e_0\rangle\]
Let $q={}^t\!(q_1,\ldots,q_k)$ be in the orthogonal complement of $\mathbb{C}[B]\gamma$ in $\mathbb{C}^k$. Then, we have
\[\sum_{i=1}^k \overline{q_i} \langle  P(l^*)g_i,e_0\rangle=  \left\langle \sum_{i=1}^k \overline{q_i} g_i,P^*(l)e_0\right\rangle=0\]for any polynomial $P$. Since $\{Pe_0|P\in\mathbb{C}[l]\}$ is dense in $c_0(\mathbb{Z}_{\ge 0})$, we get $\sum_{i=1}^k \overline{q_i} g_i=0$ and thus $q=0$. This implies that there are polynomials $Q_1,\ldots,Q_k$ such that $\{Q_i(B)\gamma\}_{i=1}^k$ is the canonical basis of $\mathbb{C}^k$. By using polynomials $\{P_i\}_{i=1}^k,\{Q_i\}_{i=1}^k$ which we obtained, we have
\[[B^n]_{ij}={}^t\!\lambda P_i(B)B^nQ_j(B) \gamma,\]
which implies that there are $L,M\in \mathbb{N}$ and constants $c_{ij}^{lm}\in \mathbb{C}$ and $x_l,y_m \in \mathbb{Z}_{\ge 0}$  ($1\le i,j\le k$, $1\le l\le L$, $1\le m \le M$) such that 
\[[B^n]_{ij}=\sum_{l,m}c_{ij}^{lm}a_{n+x_l+y_m}.\]

Let $B=R^{-1}\Lambda R$ be the Jordan standard form of $B$. Since $\Lambda^n=RB^nR^{-1}$, by the previous argument, each entry of $\Lambda^n$ is spanned by $\{a_{n+x_l+y_m}|1\le l\le L, \ 1\le m \le M\}$ whose linear coefficients are independent from $n$. By assumption, we have $\lim_{n\to \infty}a_{n+x_l+y_m}=0$, and all entries of $\Lambda^n$ converge to $0$ as $n\to \infty$. Let $\{\lambda_i\}_{1\le i\le k}$ be the set of  eigenvalues of $B$. Since $\lambda_i^n$ is an entry of $\Lambda^n$, we have $\lim_{n\to \infty}\lambda_i^n=0$, and $\lambda_i$ must satisfy $|\lambda_i|<1$. Since $a_n={}^t\! \lambda R^{-1} \Lambda^n R \gamma$, $a_n$ is spanned by $\{\binom{n}{j}\lambda_i^{n-j}|1\le i,j \le k\}$, we can find $M>0$ such that $|a_n|\le MA^n$ with $A=\max_{1\le i \le k}{|\lambda_i|}<1$. 
\end{proof}

 \begin{lm}\label{lemma3}
  For $f\in D(s) \cap L^2(s,\tau)$ with $f(\Omega)=\sum_{w \in [d]^*}\alpha_w e_w$, there exists $M>0$ and $0<A<1$ such that, for any $n \in \mathbb{Z}_{\ge 0}$, we have
   \[ \sum_{|w|=n}|\alpha_w|^2 \le MA^n. \]
    \end{lm}
    \begin{proof}
Since the sequence $(|\alpha_{uv}|^2)_{u\in [d]^*}$ is $l^1$-summable for any $v\in[d]^*$, we can see the linear subspace $K$ in Lemma \ref{lemma2} as a subspace of $l^1([d]^*)$. Through the linear map defined by
   \[ l^1([d]^*)\ni (\beta_w)_{w \in [d]^*} \mapsto \left( \sum_{|w|=n}\beta_w\right)_{n=0}^{\infty} \in l^1(\mathbb{Z}_{\ge 0}), \]
    $K$ is mapped into the subspace of $l^1(\mathbb{Z}_{\ge 0})$ spanned by 
   \[\left\{\left(\sum_{|u|=n}|\alpha_{uv}|^2\right)_{n=0}^{\infty} \bigg| \ v \in [d]^*\right\}, \] 
   which contains the subspace of $l^1(\mathbb{Z}_{\ge 0})$ spanned by
   \[\left\{\left(\sum_{|w|=n+m}|\alpha_{w}|^2\right)_{n=0}^{\infty} \Bigg| \ m \in \mathbb{Z}_{\ge 0}\right\}. \]
   By applying Lemma \ref{keylemma} (with $a_n=\sum_{|w|=n}|\alpha_w|^2$), we can find $M>0$ and $0<A<1$ such that
   \[ \sum_{|w|=n}|\alpha_w|^2 \le MA^n. \]
   
\end{proof}

\begin{proof}[Proof of Theorem \ref{Main}]
By using the Haagerup inequality for the standard free semicircular system (see \cite[Main theorem]{MR1811255} or \cite[Lemma 3.5]{MR4452070}) and Lemma \ref{lemma3}, there exists $M'>0$ and $0<A'<1$ such that for any $n \in \mathbb{Z}_{\ge 0}$, we have 
\[ \left\|\sum_{|w|=n}\alpha_w U_w\right\|\le M'(n+1){A'}^n\]
where $U_w$ is a non-commutative polynomial in $s$ which is uniquely determined by $U_w\Omega = e_w$. 
 Therefore, we have
 \[\sum_{n=0}^{\infty}\left\|\sum_{|w|=n}\alpha_w U_w\right\|\le M'\sum_{n=0}^{\infty}(n+1){A'}^n <\infty, \]
 which implies $\sum_{n=0}^{\infty}\sum_{|w|=n}\alpha_w U_w$ converges to a bounded operator $f'\in L^{\infty}(s,\tau)$ in operator norm. Since $f'(\Omega) =f(\Omega)$, we conclude $f=f'\in L^{\infty}(s,\tau)$. 

\end{proof}
Theorem \ref{keythm} is also known for free Haar unitaries by \cite{MR1452435} and \cite{MR1742860} where they use an operator $P$ instead of right annihilation operators $\{r^*(e_i)\}_{i=1}^d$. Since the domains of $P$ and any $L^2$-bounded operators $f$ contain the algebra of $\ast$-polynomials, the proof of Lemma \ref{lemma1} also works for free Haar unitaries. Moreover, we can prove other lemmas for free Haar unitaries in a similar way (see \cite{MR1452435}), and therefore our main result also holds for free Haar unitaries. 
\begin{thm}\label{freeHaar}
   Let $u=(u_1,\ldots,u_d)$ be a tuple of the free Haar unitaries. If $f \in D(u)\cap L^2(u,\tau)$, then $f \in L^{\infty}(u,\tau)$. Moreover, for $f \in D(u)$, the usual spectrum and $L^2$-spectrum of $f$ coincide.
\end{thm}
\begin{rem}
    Since Lemma \ref{keylemma} only uses $c_0$-condition of the coefficients,
    one may wonder if we can replace the $L^2$-condition in Theorem \ref{Main} by a weaker condition. Actually, we can see that all $L^1$-bounded rational functions in a single Haar unitary are bounded since the $L^1$-condition tells us that the Fourier coefficients $\{a_n\}_{n\in \mathbb{Z}}$ converge to $0$ as $n \to \pm \infty$.
    However, we have a counterexample for a single semicircle operator $s$. By elementary computation, we can see $f(s)=(2-s)^{-1}$ is integrable (not square-integrable) with respect to the standard semicircle distribution, but $f(s)$ is not bounded since the support of $s$ is $[-2,2]$. 

    Interestingly, we also have a counterexample for $d$ free Haar unitaries $\{u_i\}_{i=1}^d$ with $d\ge 2$. Actually, $f=(2\sqrt{2d-1}-\sum_{i=1}^du_i+u_i^*)^{-1}$ is integrable for $d\ge 2$ (the spectral measure of $\sum_{i=1}^du_i+u_i^*$ is the Kesten-McKay distribution), but not bounded since the support of $\sum_{i=1}^du_i+u_i^*$ is $[-2\sqrt{2d-1},2\sqrt{2d-1}]$.
\end{rem}
In the proof of Theorem \ref{Main}, we use the Haagerup inequality. We have another application of this inequality to the spectrum, which has already been discussed in \cite[Remark 3.5]{MR4685342}.
\begin{prop}\label{l2fourmula}
    For a (non-commutative) polynomial $p \in \mathbb{C}\langle s \rangle$, we have
    \[ r(p) = \lim_{n \to \infty}\|p^n\|_2^{\frac{1}{n}}\]
    where $r(p)$ denotes the spectral radius of $p$.
\end{prop}
\begin{proof}
    Recall the formula of the spectral radius:
\[ r(p)=\lim_{n \to \infty}\|p^n\|^{\frac{1}{n}}.\]
Since we have $\|X\|_2 \le \|X\|$ for any $X \in L^{\infty}(s,\tau)$, we have
\[ \limsup_{n\to \infty }\|p^n\|_2^{\frac{1}{n}} \le  r(p). \]
On the other hand, by using the Haagerup inequality for general polynomials in the standard free semicircular system, we have
\[ \|p^n\|=(nk+1)^{\frac{3}{2}}\|p^n\|_2\]
where $k$ is the degree of $p$.
Therefore, we obtain
\[ r(p)\le \liminf_{n \to \infty} \|p^n\|_2^{\frac{1}{n}}. \]

\end{proof}
\begin{rem}
The proof of Proposition \ref{l2fourmula} works for polynomials in the algebra with the rapid decay property. In particular, any element in the free group algebra $\mathbb{C}[F_d]$ satisfies the formula in this proposition. 
    The formula also holds for matrix algebras of any size since the $L^2$-norm (with respect to a normalized matrix trace) is equivalent to the $L^{\infty}$-norm. However, this formula is not always true for any operator in a $\mathrm{II}_1$-factor. As a counterexample, we can take an operator $pu$ in the cross product von Neumann algebra $L^{\infty}(\{0,1\}^{\mathbb{Z}},\prod_{\mathbb{Z}}\frac{1}{2}(\delta_0+\delta_1))\rtimes_{\sigma} \mathbb{Z}$ with respect to the Bernoulli shift $\sigma$ where  $p$ is an indicator function (projection) on a cylinder set of $\{\omega_0=0\}$ and $u$ is the unitary operator corresponding to $1 \in \mathbb{Z}$. Indeed, we have $(pu)^n=\left(\prod_{i=1}^n \sigma^{i-1}(p)\right) u^n$ whose operator norm is $1$ (since $pu$ is a product of an orthogonal projection and unitary), but $L^2$norm is $2^{-\frac{n}{2}}$.
    The author would like to thank Adrian Ioana for pointing out this counterexample.
\end{rem}
To conclude this section, we explain a method to compute the resolvent sets of (not necessarily self-adjoint) polynomials in free semicircular operators. Let $f$ be a polynomial in free semicircular operators. By Corollary \ref{spectrum}, it is enough to find the $L^2$-resolvent set, so we assume $(\lambda-f)^{-1}\in L^2(s,\tau)$ for $\lambda \in \mathbb{C}$. Then, we can determine the condition for $\lambda$ by the following three steps:
\begin{enumerate}
    \item Set $(\lambda-f)^{-1}\Omega=\sum_{w \in [d]^*}\alpha_w e_w$, and compute $(\lambda-f)\sum_{w \in [d]^*}\alpha_w e_w$.
    \item Deduce the recursion for $\{\alpha_w\}_{w \in [d]^*}$ from $(\lambda-f)\sum_{w \in [d]^*}\alpha_w e_w=\Omega.$
    \item Deduce the recursion for $a_n=\sum_{|w|=n}|\alpha_w|^2$. Then, $\lim_{n \to \infty}a_n=0$ is an equivalent condition to $\lambda \in \mathrm{spec}(f)^c$ by Lemma \ref{keylemma} and Corollary \ref{spectrum}.
\end{enumerate}
For the last step, the condition $\lim_{n \to \infty}a_n=0$ is often equivalent to the condition that the spectral radius (i.e., the maximum of the absolute values of the eigenvalues) of the representation matrix of the recursion for $a_n$ is less than $1$. This happens when the solution of the recursion involves the eigenvalue of the representation matrix that has the largest absolute value. This depends on the initial vector of the recursion, but this method works in many cases. In the next section, we compute the spectra of holomorphic polynomials in free circular operators by using this method.

  \section{Computation of the spectrum for free circular elements}
  In this section, we compute the spectrum of polynomials in a free circular system $\{c_i\}_{i=1}^d$. A circular random variable is a free analog of a complex Gaussian random variable, and a free circular system $\{c_i\}_{i=1}^d$ can be written by using a free semicircular system  $\{s_i\}_{i=1}^{2d}$ with the same tracial state $\tau$:  
  \[ c_i= \frac{s_i+\sqrt{-1}s_{d+i}}{\sqrt{2}}.\]
  Note that $s_i=\frac{c_i+c_i^*}{\sqrt{2}}$ and $s_{d+i}=\frac{c_i-c_i^*}{\sqrt{-2}}$. This implies that the set $\mathbb{C}\langle c, c^*\rangle$ of non-commutative polynomials in $\{c_i\}_{i=1}^d\cup \{c_i^*\}_{i=1}^d $ is equal to the set  $\mathbb{C}\langle s\rangle$ of non-commutative polynomials in $\{s_i\}_{i=1}^{2d}$, and therefore their division closures $D(c,c^*)$, $D(s)$ in $L(s,\tau)=L(c,c^*,\tau)$ also coincide. In particular, for any $f \in D(c,c^*)$, the $L^2$-spectral set of $f$ is equal to the usual spectral set of $f$.
  
  To describe GNS-Hilbert space $L^2(c,c^*,\tau)$  we set for each $i=1,\ldots,d$,
  \begin{align*}
u_i&=\frac{e_i+\sqrt{-1}e_{d+i}}{\sqrt{2}}\\
u_{\overline{i}}&=\frac{e_i-\sqrt{-1}e_{d+i}}{\sqrt{2}},
  \end{align*}
  so that we have
  \[c_i = l(u_i)+l^*(u_{\overline{i}}), \quad  c_i^* = l(u_{\overline{i}})+l^*(u_i). \]
  Note that $\{u_i\}_{i=1}^d\cup \{u_{\overline{i}}\}_{i=1}^d$ is a orthonormal basis of $H=\mathbb{C}^{2d}$ and
  \[ c_i\Omega =u_i, \quad c_i^*\Omega = u_{\overline{i}}.\]
 Let $[d,\overline{d}]^*$ be the free monoid generated by $[d,\overline{d}]=\{1,\ldots,d\} \cup \{\overline{1},\ldots,\overline{d}\}$ and the identity (empty word) $\Omega$. Similar to the semicircle case, for $w=w_1\cdots w_n \in [d,\overline{d}]$, we set 
 \[ u_w = u_{w_1}\otimes \cdots \otimes u_{w_n}\]
 and $u_{\Omega}=\Omega$. Then $\{u_w| \ w \in [d,\overline{d}]\}$ forms an orthonormal basis of $\mathcal{F}(\mathbb{C}^{2d})$.  Since we have the isomorphism between $L^2(c,c^*,\tau)$ and $\mathcal{F}(\mathbb{C}^{2d})$ given by $X \mapsto X\Omega$, for each $w \in [d,\overline{d}]^*$, there is a unique polynomial $U_w\in \mathbb{C}\langle c,c^*\rangle$ such that $U_w(\Omega) =u_w$. By the definition of $\{c_i\}_{i=1}^d$, we have for each $i\in [d]$ and $w=w_1w_2 \cdots w_n\in [d,\overline{d}]^*$, 
 \begin{align*}
     u_{iw}&=c_iu_w-\delta_{\overline{i},w_1}u_{w_2\cdots w_n}\\
     u_{\overline{i}w}&=c_i^*u_w-\delta_{i,w_1}u_{w_2\cdots w_n},
 \end{align*}
 which gives the following recursion for $U_w$:
 \begin{align*}
     U_{iw}&=c_iU_w-\delta_{\overline{i},w_1}U_{w_2\cdots w_n}\\
     U_{\overline{i}w}&=c_i^*U_w-\delta_{i,w_1}U_{w_2\cdots w_n}.
 \end{align*}
 By using this recursion, we have the following formula 
 for $U_w$ in terms of $l,l^*$ (cf. \cite[Proposition 2.7]{MR1463036}):
 \[U_w=\sum_{k=0}^nl(u_{w_1})\cdots l(u_{w_k})l^*(u_{\overline{w_{k+1}}})\cdots l^*(u_{\overline{w_{n}}}) \]
 where $w=w_1\cdots w_n$ ($w_i \in [d,\overline{d}]$) and we set $\overline{(\overline{i})}=i$. 

We focus on polynomials in a free circular system $c=\{c_i\}_{i=1}^d$ without adjoints $c^*=\{c_i^*\}_{i=1}^d$ and $\mathbb{C}\langle c\rangle$ denotes the set of such polynomials (say holomorphic polynomials). Let $L^2(c,\tau)$ be the $L^2$-closure of $\mathbb{C}\langle c\rangle$. Note that $X \mapsto X\Omega$ gives the Hilbert space isomorphism between $L^2(c,\tau)$ and $\mathcal{F}(\mathbb{C}^d)$ which is spanned by $\{u_w\}_{w \in [d]^*}$, and we have $u_w=c_w\Omega$ where $c_w=c_{w_1}\cdots c_{w_n}$ for $w=w_1\cdots w_n$.

The following lemma is fundamental, but it is useful to compute the spectrum. 
\begin{lm}\label{holmorphic}
    Let $f \in \mathbb{C}\langle c\rangle$. For $\lambda \in \mathbb{C}$, we have \[ (\lambda - f)^{-1} \in L^2(c,c^*,\tau) \Rightarrow (\lambda - f)^{-1} \in L^2(c,\tau).\]
\end{lm}
\begin{proof}
Note that the right annihilation operators $\{r^*(u_{\overline{i}})\}_{i=1}^d$
satisfy for any $i,j \in [d]$
    \begin{align*}
        [r^*(u_{\overline{i}}), l(u_j)]=0, & \quad [r^*(u_{\overline{i}}),l(u_{\overline{j}})]=\delta_{ij}P_{\Omega}\\
         [r^*(u_{\overline{i}}), l^*(u_j)]=0, & \quad [r^*(u_{\overline{i}}),l^*(u_{\overline{j}})]=0,\
    \end{align*}
    where $P_{\Omega}$ is the orthogonal projection onto $\mathbb{C}\Omega$. In particular, we have $[r^*(u_{\overline{i}}), c_j]=0$.
     By the Leibniz rule of the commutator, we have 
     \[ [r^*(u_{\overline{i}}), f]=0\]
     for any $f \in \mathbb{C}\langle c\rangle$ and $i\in [d]$. Therefore, we have for any $i \in [d]^*$,
          \[ [r^*(u_{\overline{i}}), (\lambda-f)^{-1}]=-(\lambda-f)^{-1}[r^*(u_{\overline{i}}), \lambda-f](\lambda-f)^{-1}=0.\]
    Let $(\lambda-f)^{-1}\Omega=\sum_{w \in [d,\overline{d}]^*}\alpha_w u_{w}$ be the expansion of $(\lambda-f)^{-1}\in L^2(c,c^*,\tau)\cong \mathcal{F}(\mathbb{C}^{2d})$ with respect to $\{u_w| \ w \in [d,\overline{d}]\}$. 

    By applying Leibniz's rule of the commutator $[r^*(u_{\overline{i}}), \ \cdot \ ]$ to the expansion of $U_w$ in terms of $l,l^*$, we can see
    \[ [r^*(u_{\overline{i}}), U_w]u_v=\begin{cases}
        u_{w'}, \ w=w'\overline{i}v^*, {}^\exists w' \in [d,\overline{d}]\\
        0, \ \mathrm{otherwise},
    \end{cases}\]
    where we define $v^*=\overline{v_n}\cdots \overline{v_1}$ for $v=v_1 \cdots v_n$.
    By applying this formula to the $L^2$-expansion of $(\lambda-f)^{-1}$, we have for $v \in [d,\overline{d}]^*$  
    \[ [r^*(u_{\overline{i}}),(\lambda-f)^{-1}]u_{v^*}=\sum_{w \in [d]}\alpha_{w\overline{i}v} u_w.\]
    If a word $w \in [d,\overline{d}]^*$ contains a letter $\overline{i}\in \{\overline{1},\ldots,\overline{d}\}$, then we can write $w = w_1\overline i w_2$ for some words $w_1,w_2 \in [d,\overline{d}]^*$. Then we have
    \[ \langle[r^*(u_{\overline{i}}),(\lambda-f)^{-1}]u_{w_2^*}, u_{w_1}\rangle=\alpha_{w}=0,\]
    which implies $(\lambda-f)^{-1}\in L^2(c,\tau)$.
\end{proof}
Thanks to this lemma, we can compute the spectra of polynomials in $\mathbb{C}\langle c\rangle$ as if each $c_i$ is a left creation $l(u_i)$ on the full Fock space. To clarify, we have the following corollary. 
\begin{cor}
    Let $c=(c_1,\ldots,c_d)$ be a free circular system. Then, for any non-commutative polynomial $P\in \mathbb{C}\langle t\rangle$ with indeterminates $t=(t_1,\ldots,t_d)$, we have
    \[ \mathrm{spec}(P(c))=\mathrm{spec}(P(l))\]
    where $l=(l(e_1),\ldots,l(e_d))$ is a tuple of the left creation operators on $\mathcal{F}(\mathbb{C}^d)$.
\end{cor}
\begin{proof}
    If $\lambda$ belongs to the resolvent set $\mathrm{spec}(P(l))^c$ of $P(l)$, then $(\lambda-P(l))^{-1}$ exists as a bounded operator on $\mathcal{F}(\mathbb{C}^d)$. Therefore, we have the following orthonormal decomposition:
    \[(\lambda-P(l))^{-1}\Omega=\sum_{w\in [d]^*}\alpha_w e_w \]
    with $\sum_{w \in [d]^*}|\alpha_w|^2<\infty$. Since $\{c_w\}_{w \in [d]^*}$ forms an orthonormal basis of $L^2(c,\tau)$, the following operator
    \[ X=\sum_{w \in [d]^*}\alpha_w c_w\]
    is well-defined in $L^2(c,\tau)$. 
    Let $\Phi$ denote the isomorphism $\mathcal{F}(\mathbb{C}^d)\ni e_w \mapsto c_w \in L^2(c,\tau)$. Note that $\Phi(l_u e_v)=c_{uv}=c_uc_v$ where $l_u = l(e_{u_1})l(e_{u_2})\cdots l(e_{u_n})$ for $u=u_1u_2\cdots u_n$. This implies the following identity in $L^2(c,\tau)$,
    \[ 1=\Phi[(\lambda-P(l))(\lambda-P(l))^{-1}\Omega]=(\lambda-P(c))\Phi[(\lambda-P(l))^{-1}\Omega]=(\lambda-P(c))X,\]
    which implies $X=(\lambda-P(c))^{-1}$. Then, Corollary \ref{spectrum} tells that $\lambda$ belongs to the resolvent set $\mathrm{spec}(P(c))^c$ of $P(c)$. On the other hand, if $\lambda$ belongs to the resolvent set $\mathrm{spec}(P(c))^c$ of $P(c)$, then by Lemma \ref{holmorphic}, we have $(\lambda-P(c))^{-1}\in L^2(c,\tau)$. Therefore, we have the orthonormal decomposition with respect to $\{c_w\}_{w \in [d]^*}$
    \[(\lambda-P(c))^{-1}=\sum_{w \in [d]^*}\beta_w c_w \]
    with $\sum_{w \in [d]^*}|\beta_w|^2<\infty$. In a similar way to the proof of Theorem \ref{Main}, we can see that there exists $M>0$ and $0< A<1$ such that
    \[ \sum_{|w|=n} |\beta_w|^2\le MA^n. \]
    By Lemma \cite[Theorem 2.2 (b)]{MR1811255} (with $q=0$), we have
    \[ \left\|\sum_{|w|=n }\beta_w l_w\right\|\le \sqrt{\sum_{|w|=n}|\beta_w|^2}\le \sqrt{M}c^{\frac{n}{2}}.\]
    Therefore, the infinite series
    \[Y=\sum_{n=0}^{\infty}\sum_{|w|=n}\beta_w l_w \]
    converges in operator norm, and in particular, it is a bounded operator. Since we have
    \[ \Phi[(\lambda-P(l))Y(e_w)]=(\lambda-P(c))(\lambda-P(c))^{-1}(c_w)=\Phi(e_w),\]
    we get $(\lambda-P(l))Y(e_w)=e_w$ and $(\lambda-P(l))Y$ is the identity operator by continuity. We can also check $Y(\lambda-P(l))$ is the identity operator. Therefore, we have $Y=(\lambda-P(l))^{-1}$ and $\lambda \in \mathrm{spec}(P(l))^c$.
\end{proof}
We remark that this setting can be seen as a non-commutative analog of the Hardy space, and our Theorem \ref{Main} is close to the equivalence of conditions (i) and (vii) in \cite[Theorem A]{MR4302175}. We give several examples of computing the spectra of holomorphic polynomials in free circular systems. 
  \begin{exam}
      Let $f=\sum_{|w|=n} a_w c_w$ with $n \in \mathbb{N}$. Assume $(\lambda-f)^{-1}\in L^2(c,c^*,\tau)$, then by Lemma \ref{holmorphic}, $(\lambda-f)^{-1}\in L^2(c,\tau)$ and we can write $(\lambda-f)^{-1}\Omega = \sum_{w\in [d]^*}\alpha_w u_w$. Then, we have the equation 
      \begin{align*} \lambda \sum_{w\in [d]^*}\alpha_w u_w - f \sum_{w\in [d]^*}\alpha_{w} u_{w} &=\lambda \sum_{w\in [d]^*}\alpha_w u_w - \sum_{|w_1|=n}   \sum_{w_2\in [d]^*}a_{w_1}\alpha_{w_2} u_{w_1w_2}\\
      &=\Omega.
      \end{align*}
      Since $\{u_w\}_{w \in [d]^*}$ is an orthonormal, this equation tells us
      \begin{align*}
          \lambda \alpha_{\Omega}&=1,\\
          \lambda \alpha_w&=0, \quad 1\le |w|\le n-1,\\
          \lambda \alpha_w - a_{w_1}\alpha_{w_2}&=0, \quad w=w_1w_2, \quad |w_1|=n.
      \end{align*}
      From the first equation, we have $\lambda\neq 0$ and $\alpha_{\Omega}=\lambda^{-1}$. By the second equation, we have $\alpha_w=0 $ for any $w$ with $1\le |w|\le n-1$. When $|w|=n$, the third equation tells us \[\alpha_w=\frac{1}{\lambda}a_w \alpha_{\Omega}=\frac{1}{\lambda^2}a_w.\]
      If $n+1\le |w| \le 2n-1$ and $w=w_1w_2$ with $|w_1|=n$, then $1\le |w_2|\le n-1$. In this case, by the second equation, we have $\alpha_{w}=0$. Inductively, we have the following solution of $\{\alpha_w\}_{w \in [d]^*}$:
      \begin{align*}
          \alpha_w&=0, \quad \mathrm{when} \ |w|\neq nk \ (k \in \mathbb{Z}_{\ge 0}),\\
          \alpha_{w_1\cdots w_k}&=\frac{1}{\lambda^{k+1}}a_{w_1}a_{w_2}\cdots a_{w_k}, \quad \mathrm{when} \ |w_1|=|w_2|=\cdots=|w_k|=n.
      \end{align*}
      Since $(\lambda-f)^{-1}\in L^2(c,\tau)$ if and only if  
      \begin{align*}
          \sum_{w \in [d]^*}|\alpha_w|^2&= \sum_{k=0}^{\infty}\frac{1}{|\lambda|^{2k+2}}\sum_{w_1,\ldots,w_k \in [d]^*} \prod_{i=1}^k |a_{w_i}|^2\\
          &=\frac{1}{|\lambda|^2}\sum_{k=0}^{\infty}\left(\frac{\|f\|_2}{|\lambda|}\right)^{2k}
      \end{align*}
      is finite, we have $|\lambda|> \|f\|_2$. Therefore, by Corollary \ref{spectrum}, the spectral set of $f$ is the closed disk with radius $\|f\|_2$. 
      \end{exam}
      In the previous example, we actually see the coincidence of the spectrum and the support of the Brown measure.
      \begin{prop} For $n\in \mathbb{N}$,  $f=\sum_{|w|=n}a_w c_w$ is $\mathscr{R}$-diagonal, and we have
      \[ \mathrm{spec}(f) = \mathrm{supp} \mu_f=\{\lambda \in \mathbb{C}| \  |\lambda| \le \|f\|_2 \}\]
      where $\mu_f $ is the Brown meausure of $f$. 
      \end{prop}
  
\begin{proof}
The proof of $\mathscr{R}$-diagonality is almost the same as the proof of the well-known fact that if $a$ is $\mathscr{R}$-diagonal, then $a^n$ ($n \in \mathbb{N}$) is also $\mathscr{R}$-diagonal (\cite[Proposition 15.22.]{MR2266879}). When we compute the free cumulant $\kappa_m(\cdots,a^n,a^n,\cdots)$ and $\kappa_m(\cdots,a^{*n},a^{*n},\cdots)$ by using the product formula \cite[Theorem 11.12.]{MR2266879}, each term must contain a factor $\kappa_{m_1}(\cdots,a,a,\cdots)$ or $\kappa_{m_2}(\cdots,a^*,a^*,\cdots)$ which is equal to $0$ due to $\mathscr{R}$-diagonality of $a$. Similarly, when we compute the free cumulant $\kappa_m(\cdots,f^{n},f^{n},\cdots)$) and by expanding the sum and using the product formula, each term must contain a factor $\kappa_{m_1}(\cdots,c_i,c_j,\cdots)$ or $\kappa_{m_2}(\cdots,c_i^*,c_j^*,\cdots)$ which is equal to $0$ due to the freeness of $\{c_i\}_{i=1}^d$ and $\mathscr{R}$-diagonality of each $c_i$ (since we only use freeness and $\mathscr{R}$-diagonality, this argument works for any tuple of free $\mathscr{R}$-diagonal elements).  
The remaining part is a conclusion from the previous example and the result by Haagerup and Larsen (Theorem \ref{Haagerup-Larsen}). 
\end{proof}
 
\begin{exam}
    Let $k \in \mathbb{N}$ and $f=(1+tc_1)(1+tc_2)\cdots (1+tc_k)$ with $t\in \mathbb{R}$.
    We write $ (\lambda-f)^{-1}\Omega=\sum_{w \in [d]^*}\alpha_w u_w$. Note that we have 
    \[ f=\sum_{m=0}^k \sum_{i_1<i_2\cdots <i_m} t^m c_{i_1i_2\cdots i_m}, \]
    and we obtain the equation
    \[ (\lambda-1)\sum_{w \in [d]^*}\alpha_w u_w-\sum_{w \in [d]^*}\sum_{m=1}^k \sum_{i_1<i_2\cdots <i_m} t^m \alpha_w u_{i_1i_2\cdots i_mw}  = \Omega.\]
    This equation implies
    \begin{align*}
        (\lambda-1)\alpha_{\Omega}&=1, \\
        (\lambda-1)\alpha_{i_1i_2\cdots i_m w}&=\sum_{l=1}^{m}t^l\alpha_{i_{l+1}\cdots i_m w}
    \end{align*}
    for any $i_1<i_2<\cdots< i_m$ ($1\le m \le k$) and $w=w_1\cdots w_n$ with $w_1\le i_m$ (or $w=\Omega$). In particular, we have $\lambda\neq 1$. By iterating the second formula and the binomial theorem, we have
    \begin{align*}\alpha_{i_1i_2\cdots i_m w}&= \alpha_{w}\sum_{l=1}^{m}\sum_{m_1+m_2\cdots+m_l=m}\frac{t^m}{(\lambda-1)^l} \\
    &=\alpha_{w}\sum_{l=1}^{m}\binom{m-1}{l-1}\frac{t^m}{(\lambda-1)^l} \\
    &= \frac{t^m\lambda^{m-1}}{(\lambda-1)^m} \alpha_w.
    \end{align*}
    For each $w \in [d]^*$, we can write $w=w_1\cdots w_m$ with $w_i \in [d]^*$ such that each $w_i=x^{(i)}_1\cdots x_{|w_i|}^{(i)}$ is increasing, i.e. $x^{(i)}_1<\cdots <x_{|w_i|}^{(i)}$ and $x_1^{(i+1)}\le x_{|w_i|}^{(i)}$ ($i=1,\ldots,m-1$).
    For such $w$, we have
    \[ \alpha_w=\alpha_{w_1\cdots w_m}=\frac{t^{|w|}\lambda^{|w|-m}}{(\lambda-1)^{|w|+1}}.\]
    Note that $|w|-m$ means the number of increasing steps in the word $w$, i.e. the number of $i=1,\ldots,|w|$ such that $i+1$-th letter is strictly bigger than $i$-th letter.
    Let us set $a_n=\sum_{|w|=n}|\alpha_w|^2$. Then, we have for $n \ge 1$
    \[ a_n =\frac{t^2}{|\lambda-1|^{4}}{}^t\!\mathbf{1}\left(\frac{t^2M}{|\lambda-1|^2}\right)^{n-1} \mathbf{1}\]
\end{exam}
where $\mathbf{1}={}^t\!(1,1,\cdots,1)$ and $M \in M_k(\mathbb{C})$ is defined by
\[ M_{ij}=\begin{cases}
    |\lambda|^2 \ (i<j)\\
    1\quad \ (i\ge j).
\end{cases}\]
Since all entries of $M$ are non-negative, $\lim_{n \to \infty}a_n=0$ if and only if $ (t^2|\lambda-1|^{-2}M)^n \mathbf{1}$ goes to $0$. When $\lambda\neq 0$, all entries of $M$ are positive, and there exists a left eigenvector ${}^t\!v$ of $M$ corresponding to the Perron-Frobenius eigenvalue $\rho$ whose entries are positive. Since ${}^t\!v \mathbf{1}>0$, $M^n \mathbf{1}$ contains $\rho^n$. When $\lambda=0$, the eigenvalue of $M$ is $1$ and ${}^t\!\mathbf{1}M^n{}^t\!\mathbf{1}$ has order $n^{k-1}$.  
Therefore, $\lim_{n\to \infty }a_n =0$ if and only if the spectral radius of $M$ is less than $t^{-2}|\lambda-1|^2$ in this case. The characteristic polynomial of $M$ is
\[ \frac{(x-1+|\lambda|^2)^k-|\lambda|^2x^k}{1-|\lambda|^2}.\]
Then, the eigenvalues of $M$ are 
\[\left\{\frac{|\lambda|^2-1}{|\lambda|^{\frac{2}{k}}\exp(\frac{2\pi l}{k}\sqrt{-1})-1} \Bigg| \ l=0,\ldots,k-1\right\},\]
and the spectral radius of $M$ is  
\[ \frac{|\lambda|^2-1}{|\lambda|^{2/k}-1}.\]
Therefore, we have that $(\lambda-f)^{-1}\in L^2(c,\tau)$ if and only if
\[  \frac{|\lambda|^2-1}{|\lambda|^{2/k}-1}<t^{-2}|\lambda-1|^2 \]
and the spectral set is
\[ \mathrm{spec}(f)=\left\{\lambda \in \mathbb{C}\Bigg| \frac{|\lambda-1|^2(|\lambda|^{2/k}-1)}{|\lambda|^2-1}\le t^2\right\},\]
which is equal to the support of the Brown measure of $f$ by \cite[Theorem 1.27]{driver2025matrixrandomwalkslima} with $u_0=1$ and $t=k{t'}^2$.
 
  As the last example, we prove Theorem \ref{quadratic} for quadratic polynomials in two free circular operators $c_1,c_2$. Their Brown measures and weak convergence of the random matrix models are discussed in \cite{MR4492979}.  We ignore constant terms since it is just a shift of the spectrum.
  \begin{thm}\label{quadratic2}
  For a quadratic polynmial in free circular random variables $c_1,c_2$ written as $f= \sum_{i,j= 1}^2 a_{ij} c_i c_j +\sum_{i=1}^2 b_i c_i$ with $A=(a_{ij})\in M_2(\mathbb{C})$, $b={}^t\!\begin{pmatrix}
      b_1&b_2
  \end{pmatrix}\in \mathbb{C}^2$, the spectral set of $f$ is the set of $\lambda \in \mathbb{C}$ such that $\lambda=0$ or $\lambda$ satisfies \[\lim_{n\to \infty}Q_{\lambda}^ne_1\neq0,\] where $e_1={}^t\begin{pmatrix}
      1&0&0&0&0&0
  \end{pmatrix}$ and $Q_{\lambda}$ is the following $6\times 6$ matrix  :
  \[Q_{\lambda}=\begin{pmatrix}
      b_{\lambda}^* b_\lambda& \mathrm{Tr}(|A_{\lambda}|^2)& b_{\lambda}^*A_{\lambda} & {}^t\! b_{\lambda} \overline{A_{\lambda}}\\
      1 & 0& 0& 0 \\
      \overline{b_{\lambda}}&0 &0 & \overline{A_{\lambda}}\\
      b_{\lambda} & 0 & A_{\lambda} & 0
  \end{pmatrix}\]
  with $A_{\lambda}=\lambda^{-1}A$ and $b_{\lambda}=\lambda^{-1}b$.
  \end{thm}

  \begin{proof}
    Assume $(\lambda-f)^{-1}\in L^2(c,\tau)$ and we set $(\lambda-f)^{-1}\Omega=\sum_{w \in [d]^*}\alpha_w u_w$. Then, we obtain from $(\lambda-f)\sum_{w\in [d]^*}\alpha_w u_w =\Omega$ that $\lambda \neq 0$ and
    \begin{align*}
       \alpha_{\Omega}&=\frac{1}{\lambda}, \quad \alpha_1=\frac{b_1}{\lambda^2}, \quad \alpha_2= \frac{b_2}{\lambda^2}, \\
       \alpha_{11w}&=\frac{1}{\lambda}(a_{11}\alpha_w+b_1\alpha_{1w}), \quad \alpha_{12w}=\frac{1}{\lambda}(a_{12}\alpha_w+b_1\alpha_{2w}), \\
      \alpha_{21w}&=\frac{1}{\lambda}(a_{21}\alpha_w+b_2\alpha_{1w}), \quad \alpha_{22w}=\frac{1}{\lambda}(a_{22}\alpha_w+b_2\alpha_{2w}).
    \end{align*}
      Let us set \[x_n=\sum_{|w|=n}|\alpha_w|^2, \quad y_n=\sum_{|w|=n} \alpha_{w}\overline{\alpha_{1w}},\quad z_n=\sum_{|w|=n} \alpha_{w}\overline{\alpha_{2w}}.\]
      Then, we get the recursion
      \begin{align*}
          x_{n+2}&=\sum_{i,j=1}^2\sum_{|w|=n}|\alpha_{ijw}|^2\\&=\frac{1}{|\lambda|^2}\sum_{i,j=1}^2\sum_{|w|=n}|a_{ij}\alpha_{w}+b_i\alpha_{jw}|^2\\
          &=(b_{\lambda}^*b_{\lambda})x_{n+1}+\mathrm{Tr}(|A_{\lambda}|^2)x_n+b_{\lambda}^*A_{\lambda}\begin{pmatrix}
              y_n\\ z_n
          \end{pmatrix}+{}^t\!b_{\lambda}\overline{A_{\lambda}}\begin{pmatrix}
              \overline{y_n}\\
              \overline{z_n}
          \end{pmatrix}.
      \end{align*}
      We also have
      \begin{align*}
          y_{n+1}&=\sum_{i=1}^2\sum_{|w|=n}\alpha_{iw}\overline{\alpha_{1iw}}\\
          &=\frac{1}{\overline{\lambda}}\sum_{i=1}^2\sum_{|w|=n}\alpha_{iw}(\overline{a_{1i}\alpha_w+b_1\alpha_{iw}})\\
          &=\overline{\left(\frac{b_1}{\lambda}\right)}x_{n+1}+\overline{\left(\frac{a_{11}}{\lambda}\right)}\overline{y_n}+\overline{\left(\frac{a_{12}}{\lambda}\right)}\overline{z_n},
      \end{align*}
      and similarly,
       \begin{align*}
          z_{n+1}&=\sum_{i=1}^2\sum_{|w|=n}\alpha_{iw}\overline{\alpha_{2iw}}\\
          &=\frac{1}{\overline{\lambda}}\sum_{i=1}^2\sum_{|w|=n}\alpha_{iw}(\overline{a_{2i}\alpha_w+b_2\alpha_{iw}})\\
          &=\overline{\left(\frac{b_2}{\lambda}\right)}x_{n+1}+\overline{\left(\frac{a_{21}}{\lambda}\right)}\overline{y_n}+\overline{\left(\frac{a_{22}}{\lambda}\right)}\overline{z_n}.
      \end{align*}
      This implies
      \[ \begin{pmatrix}
          y_n\\z_n
      \end{pmatrix}=\overline{b_{\lambda}}x_{n+1}+\overline{A_{\lambda}}\begin{pmatrix}
          \overline{y_n}\\ \overline{z_n}
      \end{pmatrix}.\]
      Since the initial condition for this recursion is
      \begin{align*}
          x_0&=\frac{1}{|\lambda|^2}, \quad x_1=\frac{1}{|\lambda|^2}b_{\lambda}^*b_{\lambda}, \\
          y_0&=\overline{\left(\frac{b_1}{|\lambda|^2\lambda}\right)},\quad z_0=\overline{\left(\frac{b_2}{|\lambda|^2\lambda}\right)},
      \end{align*}
      we have for $n \ge 0$,
      \[ \begin{pmatrix}x_{n+1}\\ x_n \\ y_n\\z_n \\ \overline{y_n}\\ \overline{z_n}\end{pmatrix}=\frac{1}{|\lambda|^2} Q_{\lambda}^{n+1} e_1
      \]
      where $e_1={}^t\!(1,0,0,0,0,0)$. 
      Note that by the Cauchy-Schwarz inequality, $|y_n|$ and $|z_n|$ are bounded by $\sqrt{x_{n}x_{n+1}}$ which goes to $0$ as $n \to \infty$. Therefore, we must have $\lim_{n\to \infty}Q_{\lambda}^{n+1} e_1=0$. On the other hand, if $\lim_{n\to \infty}Q_{\lambda}^{n+1} e_1=0$, then $\lim_{n\to \infty}x_n=0$ and $\|(\lambda-f)^{-1}\|_2^2=\sum_{n=0}^{\infty}x_n <\infty$ by Lemma \ref{keylemma}. Therefore, $\lambda \in \mathrm{spec}^2_1(f)^c=\mathrm{spec}(f)^c$.
  \end{proof}
  
  \begin{rem}\label{equivalence_spectral2}
     As we mentioned in Remark \ref{equiv_sepctral1}, there is a possibility that the spectral radius of $ r(Q_{\lambda})\ge 1$ doesn't imply $\lim_{n\to \infty}Q_{\lambda}^ne_1 \neq 0$, but we can prove the equivalence in several cases.  
    To see this, let $\rho=\rho(\lambda)$ be an eigenvalue of $Q_{\lambda}$ such that $e_1$ has no component from the generalized eigenspaces $V_{\rho}$ when we decompose $\mathbb{C}^6$ into the generalized eigenspaces of $Q_{\lambda}$. The problem happens when $|\rho|$ is the largest among all absolute values of eigenvalues of $Q_{\lambda}$. For such $\rho$ (we may assume $\rho\neq 0$) and any left eigenvector ${}^t\!v$ of $Q_{\lambda}$ corresponding to $\rho$, we must have ${}^t\!ve_1=0$. Let us set $ {}^t\!v=\begin{pmatrix}
        v_1&v_2&v_3&v_4&v_5&v_6
    \end{pmatrix}$. Then, we have ${}^t\!ve_1=v_1=0$. Since ${}^t\!vQ_{\lambda}=\rho {}^t\!v$, we have $\rho v_2 = \mathrm{Tr}(|A_{\lambda}|^2)v_1=0$ and thus $v_2=0$. Therefore ${}^t\!v$ must satisfy
    \begin{align*}
\begin{pmatrix}
    v_3 & v_4
\end{pmatrix} \overline{b_{\lambda}}+ \begin{pmatrix}
    v_5 & v_6
\end{pmatrix}b_{\lambda}&=0,\\
\begin{pmatrix}
    v_3 & v_4 & v_5 & v_6
\end{pmatrix} \begin{pmatrix}
    0 & \overline{A_{\lambda}}\\
    A_{\lambda} & 0
\end{pmatrix}&=\rho \begin{pmatrix}
    v_3 & v_4 & v_5 & v_6
\end{pmatrix}.
    \end{align*}
  
  In particular, $\rho$ should be an eigenvalue of $\begin{pmatrix}
      0 & \overline{A_{\lambda}} \\
      A_{\lambda}& 0
  \end{pmatrix}$.
  A typical case where such ${}^t \! v$ exists is when $b=0$. In this case, $Q_{\lambda}$ is equal to
 \[ \begin{pmatrix}
      0 & \mathrm{Tr}(|A_{\lambda}|^2) & 0 & 0\\
      1 & 0 & 0 & 0\\
      0 & 0 & 0 & \overline{A_{\lambda}}\\
      0  & 0 & A_{\lambda} & 0
  \end{pmatrix},\]
  and $\sqrt{\mathrm{Tr}(|A_{\lambda}|)}$ is an eigenvalue which is greater than or equal to $|\rho|$, so $e_1$ involves an eigenvalue which is greater than or equal to $|\rho|$. 
  
  If $\rho$ is an eigenvalue of $\begin{pmatrix}
      0 & \overline{A_{\lambda}} \\
      A_{\lambda}& 0
  \end{pmatrix}$, then $\rho^2$ is an eigenvalue of $A_{\lambda}\overline{A_{\lambda}}$ and $\overline{A_{\lambda}}A_{\lambda}$. Let $\rho_1$ and $\rho_2$ be eigenvalues of $A_{\lambda}\overline{A_{\lambda}}$ (and $\overline{A_{\lambda}}A_{\lambda}$) with $|\rho_1|\ge|\rho_2|$. Then, we have $\rho_1=\overline{\rho_2}$ or $\rho_1,\rho_2$ are real numbers with the same sign since $\rho_1\rho_2=|\det(A_{\lambda})|^2$. If $\rho_1=\overline{\rho_2}$, then we have $|\rho_1|^2=\rho_1\rho_2=|\det(A_{\lambda})|^2$, so $ |\rho_1|=|\det(A_{\lambda})|$. On the other hand, we can see $\det(Q_{\lambda})=-\mathrm{Tr}(|A_{\lambda}|^2)|\det(A_{\lambda})|^2$. Since $|\rho|$ is largest and $|\rho|^2=|\rho_1|$, we should have
  \[ \mathrm{Tr}(|A_{\lambda}|^2)|\det(A_{\lambda})|^2\le |\rho|^6=|\rho_1|^3=|\det(A_{\lambda})|^3.\] However, $|\det(A_{\lambda})|<\mathrm{Tr}(|A_{\lambda}|^2)$ for $A_{\lambda}\neq 0$, which is a contradiction. Therefore, $\rho_1$ and $\rho_2$ should be different real numbers. Since $ \overline{A_{\lambda}}A_{\lambda}=\frac{1}{|\lambda|^2}\overline{A}A$, a counterexample requires that $\overline{A}A$ have different real eigenvalues. Note that by the same reason (comparison of determinants), $\rho^2$ should be equal to $\rho_1$, not $\rho_2$. Moreover, if $x$ is an eigenvector of $\overline{A_{\lambda}}A_{\lambda}$ corresponding to $\rho_1$, then, $\overline{A_{\lambda}}\overline{x}$ is also an eigenvector of $\overline{A_{\lambda}}A_{\lambda}$ corresponding to $\rho_1$ since $\rho_1$ is a real number. Since the dimensions of eigenspaces of $\overline{A_{\lambda}}A_{\lambda}$ are $1$, there exists $\alpha \in \mathbb{C}$ such that $\overline{A_{\lambda}}\overline{x}=\alpha x$. By multiplying both side by $ A_{\lambda}$, we have 
  \[A_{\lambda}\overline{A_{\lambda}}\overline{x}=\alpha A_{\lambda}x=|\alpha|^2\overline{x}. \] Since $A_{\lambda}\overline{A_{\lambda}}\overline{x}=\rho_1 \overline{x}$, we have $ \rho_1=|\alpha|^2> 0$. Therefore, to have a counterexample, we may assume $\rho_1>\rho_2\ge 0$ and eigenvalues of $\begin{pmatrix}
      0 & \overline{A_{\lambda}} \\
      A_{\lambda}& 0
  \end{pmatrix}$ are $\pm \sqrt{\rho_1}$ and $\pm \sqrt{\rho_2}$, and $\rho=\pm\sqrt{\rho_1}$. 
  
  Now, we assume $A_{\lambda}$ is symmetric. Then, $A$ is also symmetric and $A_{\lambda}\overline{A_{\lambda}}=A_{\lambda}A_{\lambda}^*$ is a self-adjoint matrix. Let $b_1, b_2$ be eigenvectors of $A_{\lambda}\overline{A_{\lambda}}$ corresponding to $\rho_1,\rho_2$ such that $b_{\lambda}=b_1+b_2$. Since $\begin{pmatrix}
      v_5&v_6
  \end{pmatrix}$ is a left eigenvector of $A_{\lambda}\overline{A_{\lambda}}$ corresponding to $\rho^2=\rho_1$ (recall that $(v_1,..,v_6)$ is a left eigenvector of $\begin{pmatrix}
      0 & \overline{A_{\lambda}} \\
      A_{\lambda}& 0
  \end{pmatrix}$), we have $\begin{pmatrix}
      v_5&v_6
  \end{pmatrix}b_{\lambda}=\begin{pmatrix}
      v_5&v_6
  \end{pmatrix}b_1$, and similarly, we have  $\begin{pmatrix}
    v_3&v_4  
  \end{pmatrix}\overline{b_{\lambda}}=\begin{pmatrix}
      v_3&v_4
  \end{pmatrix}\overline{b_1} $. We can also see that there exist $\alpha_1,\alpha_2\in \mathbb{C}$ such that $ \overline{A_{\lambda}} b_1=\alpha_1\overline{b_1}$ and $ \overline{A_{\lambda}} b_2=\alpha_2\overline{b_2}$. By applying the condition of ${}^t\!v$, we have
  \[ \alpha_1[\begin{pmatrix}
    v_3 & v_4
\end{pmatrix} \overline{b_{\lambda}}+ \begin{pmatrix}
    v_5 & v_6
\end{pmatrix}b_{\lambda}]= \begin{pmatrix}
    v_3 & v_4
\end{pmatrix}  \overline{A_{\lambda}}b_1+\alpha_1 \begin{pmatrix}
    v_5 & v_6
\end{pmatrix}b_1= 0.\] Since we have $\begin{pmatrix}
    v_3 & v_4
\end{pmatrix}  \overline{A_{\lambda}}=\rho\begin{pmatrix}
    v_5 & v_6
\end{pmatrix}$, we obtain $ \alpha_1=-\rho$ or $\begin{pmatrix}
    v_3 & v_4
\end{pmatrix} \overline{b_1}=\begin{pmatrix}
    v_5 & v_6
\end{pmatrix}b_1=0$. If $\begin{pmatrix}
    v_3 & v_4
\end{pmatrix} \overline{b_1}=\begin{pmatrix}
    v_5 & v_6
\end{pmatrix}b_1=0$, $b_1$ should be an eigenvector of $ A_{\lambda}\overline{A_{\lambda}}$ corresponding to $\rho_2$, so $b_1=0$. For $x \not\in \{\pm \sqrt{\rho_1},\pm\sqrt{\rho_2}\}$,
we compute the determinant by the Schur complement formula, and we have
\[ \det(xI_6-Q_{\lambda})=\det\begin{pmatrix}
    xI_2 & -\overline{A_{\lambda}}\\
    -A_{\lambda}& xI_2
\end{pmatrix} \det(M_1-M_2) \]
where

 \[ M_1=\begin{pmatrix}
    x-b_{\lambda}^*b_{\lambda} & -\mathrm{Tr}(|A_{\lambda}|^2) \\
    -1 & x
\end{pmatrix},\]
  and 
  \[M_2= \begin{pmatrix}
    {}^t\!b_{\lambda} & b_{\lambda}^* \\ 0 & 0
\end{pmatrix}\begin{pmatrix}
      0 & \overline{A_{\lambda}} \\
      A_{\lambda}& 0
  \end{pmatrix}\begin{pmatrix}
       xI_2 & -\overline{A_{\lambda}}\\
    -A_{\lambda}& xI_2
  \end{pmatrix}^{-1}\begin{pmatrix}
    \overline{b_{\lambda}} & 0 \\ b_{\lambda} & 0
\end{pmatrix}=\begin{pmatrix}
    M_3 & 0 \\ 0 & 0
\end{pmatrix}\]
with 
\begin{align*}
    M_3&=\begin{pmatrix}
    {}^t\!b_{\lambda} & b_{\lambda}^* 
\end{pmatrix}\begin{pmatrix}
      0 & \overline{A_{\lambda}} \\
      A_{\lambda}& 0
  \end{pmatrix}\begin{pmatrix}
       xI_2 & -\overline{A_{\lambda}}\\
    -A_{\lambda}& xI_2
  \end{pmatrix}^{-1}\begin{pmatrix}
    \overline{b_{\lambda}} \\ b_{\lambda} 
\end{pmatrix}\\&= -2b_{\lambda}^*b_{\lambda}+x\begin{pmatrix}
    {}^t\!b_{\lambda} & b_{\lambda}^* 
\end{pmatrix}\begin{pmatrix}
       xI_2 & -\overline{A_{\lambda}}\\
    -A_{\lambda}& xI_2
  \end{pmatrix}^{-1}\begin{pmatrix}
    \overline{b_{\lambda}} \\ b_{\lambda} 
\end{pmatrix}.
\end{align*} 
Therefore, we obtain
\[ \det(M_1-M_2)=x^2+b_{\lambda}^*b_{\lambda}x-\mathrm{Tr}(|A_{\lambda}|^2)-x^2\begin{pmatrix}
    {}^t\!b_{\lambda} & b_{\lambda}^* 
\end{pmatrix}\begin{pmatrix}
       xI_2 & -\overline{A_{\lambda}}\\
    -A_{\lambda}& xI_2
  \end{pmatrix}^{-1}\begin{pmatrix}
    \overline{b_{\lambda}} \\ b_{\lambda} 
\end{pmatrix}.\]
Note that we have
\[\begin{pmatrix}
    xI_2 & -\overline{A_{\lambda}}\\
    -A_{\lambda}& xI_2
\end{pmatrix}^{-1}=\begin{pmatrix}
    x(x^2I_2-\overline{A_{\lambda}}A_{\lambda})^{-1} & (x^2I_2-\overline{A_{\lambda}}A_{\lambda})^{-1}\overline{A_{\lambda}}\\
   (x^2I_2-A_{\lambda}\overline{A_{\lambda}})^{-1} A_{\lambda}& x(x^2I_2-A_{\lambda}\overline{A_{\lambda}})^{-1}
\end{pmatrix}. \]
When $b_1=0$, by using the relation $ \overline{A_{\lambda}}b_2=\alpha_2\overline{b_2}$, we have
\begin{align*}
    \begin{pmatrix}
    {}^t\!b_2 & b_2^* 
\end{pmatrix}\begin{pmatrix}
       xI_2 & -\overline{A_{\lambda}}\\
    -A_{\lambda}& xI_2
  \end{pmatrix}^{-1}\begin{pmatrix}
    \overline{b_2} \\ b_2 
\end{pmatrix}&=\begin{pmatrix}
    {}^t\!b_2 & b_2^* 
\end{pmatrix}
\begin{pmatrix}
        (x+\alpha_2)(x^2-\rho_2)^{-1}\overline{b_2}\\
        (x+\overline{\alpha_2})(x^2-\rho_2)^{-1}b_2
    \end{pmatrix}\\
    &=(2x+\alpha_2+\overline{\alpha_2})(x^2-\rho_2)^{-1}b_2^*b_2.
\end{align*}
Therefore, for $x \in \mathbb{R}\setminus \{\pm\sqrt{\rho_2}\}$, we have
\begin{align*}
    \det(M_1-M_2)&=x^2+b_2^*b_2x-\mathrm{Tr}(|A_{\lambda}|^2)-x^2(2x+\alpha_2+\overline{\alpha_2})(x^2-\rho_2)^{-1}b_2^*b_2\\
    &=x^2-\mathrm{Tr}(|A_{\lambda}|^2)-x(x^2+(\alpha_2+\overline{\alpha_2})x+\rho_2)(x^2-\rho_2)^{-1}b_2^*b_2\\
    &=x^2-\mathrm{Tr}(|A_{\lambda}|^2)-x|x+\alpha_2|^2(x^2-\rho_2)^{-1}b_2^*b_2.
\end{align*} 
where we use  $|\alpha_2|^2=\rho_2$ obtained from $\overline{A_{\lambda}}b_2=\alpha_2\overline{b_2}$ by multiplying $ A_{\lambda}$.
Since $\mathrm{Tr}(|A_{\lambda}|^2)=\rho_1+\rho_2$, we have $\det(M_1-M_2)<0$ at $x=\sqrt{\rho_1}$. On the other hand, $\det(M_1-M_2)$ goes to infinity as $x \to \infty$. Therefore, there exists $x > \sqrt{\rho_1}$ such that $\det(M_1-M_2)=0$ and this $x$ becomes an eigenvalue of $Q_{\lambda}$ greater than $|\rho|=\sqrt{\rho_1}$. When $\alpha_1=-\rho$, then the relations, $\overline{A_{\lambda}} b_1=\alpha_1\overline{b_1}$ and $ \overline{A_{\lambda}} b_2=\alpha_2\overline{b_2}$, tell us
\begin{align*}
  &  \begin{pmatrix}
    {}^t\!b_{\lambda} & b_{\lambda}^* 
\end{pmatrix}\begin{pmatrix}
       xI_2 & -\overline{A_{\lambda}}\\
    -A_{\lambda}& xI_2
  \end{pmatrix}^{-1}\begin{pmatrix}
    \overline{b_{\lambda}} \\ b_{\lambda} 
\end{pmatrix}\\&=\begin{pmatrix}
    {}^t\!b_{\lambda} & b_{\lambda}^* 
\end{pmatrix}
\begin{pmatrix}
      (x+\alpha_1)(x^2-\rho_1)^{-1}\overline{b_1}+  (x+\alpha_2)(x^2-\rho_2)^{-1}\overline{b_2}\\
        (x+\overline{\alpha_1})(x^2-\rho_1)^{-1}b_1+(x+\overline{\alpha_2})(x^2-\rho_2)^{-1}b_2
    \end{pmatrix}\\
    &=\begin{pmatrix}
    {}^t\!b_{\lambda} & b_{\lambda}^* 
\end{pmatrix}
\begin{pmatrix}
      (x+\rho)^{-1}\overline{b_1}+  (x+\alpha_2)(x^2-\rho_2)^{-1}\overline{b_2}\\
        (x+\rho)^{-1}b_1+(x+\overline{\alpha_2})(x^2-\rho_2)^{-1}b_2
    \end{pmatrix}\\
    &=  2(x+\rho)^{-1}b_1^*b_1+  (2x+\alpha_2+\overline{\alpha_2})(x^2-\rho_2)^{-1}b_2^*b_2
\end{align*}
where we use the assumption $\rho=\pm \sqrt{\rho_1}$ and $b_1, b_2$ are orthogonal (since $A_{\lambda}\overline{A_{\lambda}}$ is self-adjoint). Therefore, for $x \in \mathbb{R}\setminus \{-\rho, \pm\sqrt{\rho_2}\}$, $\det(M_1-M_2)$ is equal to
\[ x^2-\mathrm{Tr}(|A_{\lambda}|^2)-x(x-\rho)(x+\rho)^{-1}-x|x+\alpha_2|^2(x^2-\rho_2)^{-1}b_2^*b_2.\]
   If $\rho=-\sqrt{\rho_1}$, then $\det(M_1-M_2)$ goes to $-\infty$ as $x \to \sqrt{\rho_1}+0$ and there exists $x>\sqrt{\rho_1}$ such that $\det(M_1-M_2)$ is equal to $0$. If $\rho=\sqrt{\rho_1}$, then at $x=\sqrt{\rho_1}$, $\det(M_1-M_2)$ is equal to
   \[ -\rho_2-\sqrt{\rho_1}|\sqrt{\rho_1}+\alpha_2|^2(\rho_1-\rho_2)^{-1}b_2^*b_2.\] Unless $\rho_2=0$ and $b_2^*b_2=0$, there exists $x>\sqrt{\rho_1}$ such that $\det(M_1-M_2)$ is equal to $0$. Finally, if $\rho_2=0$ and $b_2^*b_2=0$, then we can check that the row vector $\begin{pmatrix}
       1& \rho & -\frac{1}{2}{}^t\!b_1 & -\frac{1}{2}b_1^* 
   \end{pmatrix} $ is a left eigenvector of $Q_{\lambda}$ corresponding to $\rho$ by using the symmetry of $A$ and the relations $b_{\lambda}=b_1$ and $\overline{A_{\lambda}}b_1=-\rho \overline{b_1}$. In this case $e_1$ has a component from the generalized eigenspace of $\rho$. 
   
   To summarize these arguments, we can replace the condition $\lim_{n\to \infty}Q_{\lambda}^ne_1 \neq 0$ in Theorem \ref{quadratic2} by $ r(Q_{\lambda})\ge 1$ if one of the following holds:
  \begin{itemize}
      \item $b=0$. 
      \item $\overline{A}A$ doesn't have different real eigenvalues.
      \item $A$ is symmetric.
  \end{itemize}
  
  \end{rem}
  \subsection*{Acknowledgements}
  The author is grateful to Tao Mei for raising the question about the equivalence between $L^p$-integrability and $L^{\infty}$-boundedness for non-commutative rational functions. The author gratefully acknowledges the hospitality of Todd Kemp during the author’s postdoctoral stay at the University of California, San Diego. 
  The author is also grateful to Ching Wei Ho, Adrian Ioana, Ping Zhong, and Junichiro Matsuda for useful discussions and references. The author acknowledges support from JSPS Overseas Research Fellowships.
  \bibliographystyle{amsalpha}
\bibliography{reference.bib}
\end{document}